\documentclass[leqno,11pt]{amsart}
\usepackage{amsmath,amstext,amssymb,amsopn,amsthm,mathrsfs,wasysym}

\theoremstyle{plain}

\allowdisplaybreaks

\newtheorem{thm}{Theorem}[section]
\newtheorem{cor}[thm]{Corollary}
\newtheorem{pro}[thm]{Proposition}
\newtheorem{ob}[thm]{Observation}
\newtheorem{lem}[thm]{Lemma}
\newtheorem{rem}[thm]{Remark}
\newtheorem*{ex*}{Example}

\newcommand{\N}{\mathbb{N}}
\newcommand{\C}{\mathbb{C}}

\def\P{\mathcal P}
\def\H{\mathbb H}
\def\m{\mu} 						
\def\J{\mathcal J}

\def\l{\lambda}

\def\a{\alpha}
\def\b{\beta}
\def\ab{\alpha,\beta}
\def\t{\theta}
\def\vp{\varphi}
\def\st{\sin \frac{\theta}{2}}
\def\svp{\sin \frac{\varphi}{2}}
\def\ct{\cos \frac{\theta}{2}}
\def\cvp{\cos \frac{\varphi}{2}}
\def\Puv{\Psi^{\alpha,\beta}(t,\theta,\varphi,u,v)}

\def\Peeuv{\Psi^{\alpha,\beta}_{E}(t,\theta,\varphi,u,v)}
\def\Peeu{\Psi^{\alpha,\beta}_{E}(t,\theta,\varphi,u,1)}
\def\Peev{\Psi^{\alpha,\beta}_{E}(t,\theta,\varphi,1,v)}
\def\Pee{\Psi^{\alpha,\beta}_{E}(t,\theta,\varphi,1,1)}
\def\pia{d\Pi_{\alpha}(u)}
\def\pib{d\Pi_{\beta}(v)}
\def\piu{d\Pi_{-1\slash 2}(u)}
\def\piv{d\Pi_{-1\slash 2}(v)}

\def\oa{\Pi_{\alpha}(u)\,du}

\def\q{\mathfrak{q}}

\DeclareMathOperator{\support}{supp}

\def\Plam{\Psi^{\lambda}(t,\mathfrak{q})}
\def\Ptilde{\widetilde{\Psi}^{\lambda}(t,\mathfrak{q})}
\def\Ptildej{\widetilde{\Psi}^{\lambda - 1 + \sum_i j_i}(t,\mathfrak{q})}
\def\piK{d\Pi_{\alpha, K}(u)}
\def\piR{d\Pi_{\beta, R}(v)}
\def\abpiK{\Pi_{\alpha, K}}
\def\abpiR{\Pi_{\beta, R}}

\def\lam{\lambda}
\def\evenM{\sum_{\text{even} \, i} j_i}
\def\oddM{\sum_{\text{odd} \, i} j_i}
\def\evenN{\sum_{\text{even} \, i} j_i}
\def\oddN{\sum_{\text{odd} \, i} j_i}

\title[Analysis in the Jacobi setting]{Analysis related to all admissible type parameters\\ 
in the Jacobi setting}

\author[A. Nowak]{Adam Nowak}
\address{Adam Nowak, \newline
			Institute of Mathematics,
		Polish Academy of Sciences, \newline
      \'Sniadeckich 8,
      00--956 Warszawa, Poland    
      }
\email{anowak@impan.pl}

\author[P. Sj\"ogren]{Peter Sj\"ogren}
\address{Peter Sj\"ogren, \newline
			Mathematical Sciences, University of Gothenburg \newline
Mathematical Sciences, Chalmers University of Technology \newline
SE-412 96 G\"oteborg, Sweden 
      }
\email{peters@chalmers.se}

\author[T.Z. Szarek]{Tomasz Z. Szarek}
\address{Tomasz Z. Szarek,     \newline
			Institute of Mathematics,
		Polish Academy of Sciences, \newline
      \'Sniadeckich 8,
      00--956 Warszawa, Poland
      }
\email{szarektomaszz@gmail.com}

\begin{document}

\begin{abstract}
We derive an integral representation for the Jacobi-Poisson kernel valid for all admissible
type parameters $\ab$ in the context of Jacobi expansions. This enables us to develop a technique
for proving standard estimates in the Jacobi setting, which works for all possible $\a$ and $\b$.
As a consequence, we can prove that several fundamental operators in the harmonic analysis of
Jacobi expansions are (vector-valued) Calder\'on-Zygmund operators in the sense of the associated space 
of homogeneous type, and hence their mapping properties follow from the general theory.
The new Jacobi-Poisson kernel representation also leads to sharp estimates of this
kernel. The paper generalizes methods and results existing in the literature, but valid
or justified only for a restricted range of $\a$ and $\b$.
\end{abstract}

\maketitle

\footnotetext{
%42C05 - Orthogonal functions and polynomials, general theory
%42C10 - Fourier series in special orthogonal functions
		
		The first author was supported in part by MNiSW Grant N N201 417839.
}

%%%%%%%%%%%%%%%%%%%%%%%%%%%%%%%%%%%%%%%%%%%%%%%%%%%%%%%%%%%%%%%%%%%%%%%%%%%%%%%%%%%%
\section{Introduction} \label{sec:intro}
%%%%%%%%%%%%%%%%%%%%%%%%%%%%%%%%%%%%%%%%%%%%%%%%%%%%%%%%%%%%%%%%%%%%%%%%%%%%%%%%%%%%%

This paper is a continuation and completion of the research performed recently in \cite{NoSj}
by the first and second authors. Given parameters $\ab > -1$, consider the Jacobi
differential operator
$$
\J^{\ab} = - \frac{d^2}{d\t^2} - \frac{\a-\b+(\a+\b+1)\cos\t}{\sin \t}
	\frac{d}{d\t} + \Big( \frac{\a+\b+1}{2}\Big)^2
$$
on the interval $[0,\pi]$ equipped with the (doubling) measure
$$
d\m_{\ab}(\t) = \Big( \sin\frac{\t}2 \Big)^{2\a+1} 
	\Big( \cos\frac{\t}2\Big)^{2\b+1} \,d\t.
$$
This operator, acting initially on $C_c^2(0,\pi)$, has a natural self-adjoint extension
in $L^2(d\m_{\ab})$, whose spectral decomposition is discrete and given by the classical
Jacobi polynomials. Various aspects of harmonic analysis related to the Jacobi
setting has been studied in the literature. This line of research goes back to
the seminal work of B. Muckenhoupt and E.M. Stein \cite{MuS}, in which the ultraspherical
case ($\alpha=\beta$) was investigated. Later several other authors contributed to the subject,
see \cite[Section 1]{NoSj} and also the end of \cite[Section 2]{NoSj} for a detailed account
and references.

The main result of \cite{NoSj} is restricted to $\ab \ge -1/2$. It states that several fundamental
operators in the harmonic analysis of Jacobi expansions, including Riesz transforms, imaginary
powers of the Jacobi operator, the Jacobi-Poisson semigroup maximal operator and Littlewood-Paley-Stein
type square functions, are (vector-valued) Calder\'on-Zygmund operators in the sense of the space
of homogeneous type $([0,\pi],d\m_{\ab},|\cdot|)$. Here $|\cdot|$ stands for the ordinary distance.
Consequently, the mapping properties of these operators follow from the general theory.
The proofs in \cite{NoSj} rely on an integral formula for the Jacobi-Poisson kernel,  derived
in \cite{NoSj}
from a product formula for Jacobi polynomials due to Dijksma and Koornwinder \cite{DK}.
Unfortunately, the latter result is not valid if either $\a < -1/2$ or $\b < -1/2$, and this limitation
is inherited by the abovementioned Jacobi-Poisson kernel representation. Thus the technique
of proving estimates for kernels defined via the Jacobi-Poisson kernel developed in \cite{NoSj}
is designed for the case $\ab \ge -1/2$. The object of the present paper is to eliminate this restriction 
in the parameter values, which will require some new techniques.

Our method starts with the deduction of a suitable
integral representation of the Jacobi-Poisson kernel, valid for all $\ab > -1$, see
Proposition \ref{prop:JP}. This extended Jacobi-Poisson kernel formula contains as
a special case that obtained in \cite[Proposition 4.1]{NoSj} for $\ab \ge -1/2$ and is more
involved if either $\a$ or $\b$ is less than $-1/2$. 
Then we establish a suitable generalization to all $\ab >-1$ of the strategy employed in \cite{NoSj}
to prove standard estimates (see \eqref{gr}-\eqref{sm2} below) for kernels expressible via the 
Jacobi-Poisson kernel. To achieve this, some essentially new arguments are required, and the 
method allows a unified treatment of all parameter values $\ab > -1$.

As an application of these techniques, we prove that the maximal operator of the Jacobi-Poisson
semigroup, the Riesz-Jacobi transforms, Littlewood-Paley-Stein type square functions and multipliers of
Laplace and Laplace-Stieltjes transform type are scalar-valued or 
vector-valued Calder\'on-Zygmund operators in the sense of the space $([0,\pi],d\m_{\ab},|\cdot|)$; 
see Theorem \ref{thm:main}.
This extends to all $\ab > -1$ several results for $\ab \ge -1/2$ obtained in \cite{NoSj}
and earlier papers. Moreover, here we treat Laplace transform type multipliers, which were considered
in \cite{NoSj} only implicitly (see the comment closing \cite[Section 2]{NoSj}), and also Laplace-Stieltjes
transform type multipliers not studied earlier in the Jacobi context. These multiplier operators include, as special cases, the imaginary powers of the Jacobi operator and the related fractional
integrals (provided that $\a+\b+1\neq 0$; otherwise the bottom eigenvalue of the Jacobi operator is
$0$ and these objects are not well defined). 
Noteworthy, our technique is well suited to a wider variety of operators, including more
general forms of $g$-functions and Lusin area type integrals. In a
similar spirit, analogous problems concerning analysis for ``low'' values of type parameters were recently
investigated in the Laguerre \cite{NoSz} and Bessel \cite{CaSz} settings.

The utility of the Jacobi-Poisson kernel representation derived in Proposition \ref{prop:JP}
is also strongly supported by the fact that it makes it possible to describe the exact behavior of the kernel;
see Theorem \ref{thm:JPsharp}. The sharp estimates we prove extend to all $\ab >-1$ the bounds found recently by Nowak and Sj\"ogren \cite[Theorem 5.2]{NoSj2} under the restriction $\ab \ge -1/2$.

We remark that Wr\'obel \cite[Corollary 4.3]{Wr} recently 
proved a multiplier result for Jacobi expansions, which
is related to our multiplier results in Section \ref{sec:CZ}.
In his proof, Wr\'obel makes use of the integral representation for the Jacobi-Poisson kernel
derived in \cite[Proposition 4.1]{NoSj} (which is restricted to $\ab \ge -1/2$); see
\cite[Section 4]{Wr}. The Jacobi-Poisson kernel formula obtained in Proposition \ref{prop:JP}
should thus make it possible to extend Wr\'obel's result to a wider range of $\ab$ 
than that admitted in \cite[Corollary 4.3]{Wr};
this, however, remains to be investigated.

The paper is organized as follows. In Section \ref{sec:JP} we derive an integral representation
of the Jacobi-Poisson kernel valid for all $\ab > -1$. Section \ref{sec:prep} contains
various facts and preparatory results needed for kernel estimates. In Section \ref{sec:kerest}
we prove standard estimates for kernels associated with the operators mentioned above. 
This leads to our main results in Section
\ref{sec:CZ}, saying that the operators in question can be interpreted as Calder\'on-Zygmund
operators and giving, as a consequence, their $L^p$ mapping properties. 
Finally, Section~\ref{sec:JPsharp} is devoted to sharp estimates of the Jacobi-Poisson kernel.

Throughout the paper we use a fairly standard notation with essentially all symbols referring
to the space of homogeneous type $([0,\pi],d\m_{\ab},|\cdot|)$. Since the distance in this space is the
Euclidean one, the ball denoted $B(\theta,r)$ is simply the interval $(\theta-r,\theta+r)\cap [0,\pi]$. 
When writing estimates, we will frequently use the notation $X \lesssim Y$
to indicate that $X \le CY$ with a positive constant $C$ independent of significant quantities.
We shall write $X \simeq Y$ when simultaneously $X \lesssim Y$ and $Y \lesssim X$.

%%%%%%%%%%%%%%%%%%%%%%%%%%%%%%%%%%%%%%%%%%%%%%%%%%%%%%%%%%%%%%%%%%%%%%%%%%%%%%%%%%%%
\section{Jacobi-Poisson kernel} \label{sec:JP}
%%%%%%%%%%%%%%%%%%%%%%%%%%%%%%%%%%%%%%%%%%%%%%%%%%%%%%%%%%%%%%%%%%%%%%%%%%%%%%%%%%%%%

Let $\ab > -1$. The Jacobi-Poisson kernel is given by (see \cite[Section 2]{NoSj})
\begin{equation*} 
H_t^{\ab}(\t,\vp) = \sum_{n=0}^{\infty} e^{-t\left|n+\frac{\a+\b+1}{2}\right|}
	\P_n^{\ab}(\t)\P_n^{\ab}(\vp);
\end{equation*}
here $t>0$ and $\t,\vp \in [0,\pi]$, and $\P_n^{\ab}$ are the classical Jacobi trigonometric
polynomials, normalized in $L^2(d\m_{\ab})$.
Notice that the fraction $\frac{\a+\b+1}{2}$ may be negative.
Defining the auxiliary kernel
$$
\H_t^{\ab}(\t,\vp) := \sum_{n=0}^{\infty} e^{-t\left(n+\frac{\a+\b+1}{2}\right)}
	\P_n^{\ab}(\t)\P_n^{\ab}(\vp),
$$
the Jacobi-Poisson kernel can be written as
\begin{equation}\label{JPH}
H_t^{\ab}(\t,\vp) = \H_t^{\ab}(\t,\vp) + \chi_{\{\a+\b < -1\}} \, 2^{\a+\b+2} c_{\ab}\, \sinh\bigg( 
	\frac{\a+\b+1}{2} \, t \bigg),
\end{equation}
where
$$
c_{\ab} := \frac{\Gamma(\a+\b+2)}{2^{\a+\b+1}\Gamma(\a+1)\Gamma(\b+1)}.
$$
The last term in \eqref{JPH} is nonzero when $\a+\b < -1$.
As we shall see later,
there are important cancellations between the two terms in \eqref{JPH} for large $t$.

The kernel $\H_t^{\ab}(\t,\vp)$ can be computed explicitly by means of Bailey's formula, see
\cite[p.\,385--387]{AAR}. More precisely, we have
\begin{align}
\H_t^{\ab}(\t,\vp) & = c_{\ab}\,
	\frac{\sinh\frac{t}{2}}{\big(\cosh\frac{t}{2}\big)^{\a+\b+2}} \nonumber \\ & \quad \times
	  F_4\Bigg( \frac{\a+\b+2}{2},\frac{\a+\b+3}{2}; \a+1, \b+1; 
	\bigg(\frac{\st \svp}{\cosh\frac{t}2}\bigg)^2,
	\bigg(\frac{\ct \cvp}{\cosh\frac{t}2}\bigg)^2\Bigg), \label{HF4}
\end{align}
for $t>0$ and $\t,\vp \in [0,\pi]$. Here $F_4$ is Appel's hypergeometric function of two variables
defined by the series
$$
F_4(a_1,a_2; b_1,b_2; x,y) = \sum_{m,n=0}^{\infty} 
	\frac{(a_1)_{m+n} (a_2)_{m+n}}{(b_1)_m (b_2)_n m! n!} \, x^m y^n,
$$
where $(a)_n$ means the Pochhammer symbol, $(a)_n = a(a+1)\cdot \ldots \cdot (a+n-1)$ for $n \ge 1$
and $(a)_0 = 1$. This double power series is known to converge absolutely when 
$\sqrt{|x|}+ \sqrt{|y|} < 1$, cf. \cite[Chapter V, Section 5.7.2]{Bat}.
From this expression, the positivity of $\H_t^{\ab}(\t,\vp)$ can easily be
seen. Moreover, \eqref{HF4} provides a holomorphic extension of $\H_t^{\ab}(\t,\vp)$ as a function
of the parameters $\ab > -1$ to the region $\{(\ab) \in \C^2 : \Re \a, \Re \b > -1\}$.
Indeed, with $t>0$ and $\t,\vp \in [0,\pi]$ fixed, 
the hypergeometric series in \eqref{HF4} is a sum of holomorphic functions of $(\ab)$
converging locally uniformly in the region in question (the latter fact can be justified
by means of elementary estimates for the Pochhammer symbol). However, the formula \eqref{HF4} does not seem to be
convenient from the point of view of kernel estimates. Thus we need a more suitable representation.

In \cite[Section 4]{NoSj} the first and second authors derived the following integral representation, valid for $\a,\b \ge -1/2$ (notice that under this restriction $H_t^{\ab}(\t,\vp)$ coincides
with $\H_t^{\ab}(\t,\vp)$):
\begin{equation} \label{HI}
\H_t^{\ab}(\t,\vp) = c_{\ab} \,\sinh\frac{t}2 \iint \frac{\pia \, \pib}
	{(\cosh\frac{t}2-1 + q(\t,\vp,u,v))^{\a+\b+2}}, \qquad t>0, \quad \t,\vp \in [0,\pi],
\end{equation}
where
$$
q(\t,\vp,u,v) = 1 - u \st \svp - v \ct \cvp,
$$
and the measure $d\Pi_{\a}$ is defined in the following way. For $\a > -1/2$, we let
\begin{equation} \label{Pimes}
\Pi_{\a}(u) := \frac{\Gamma(\a+1)}{\sqrt{\pi}\Gamma(\a+1/2)} \int_0^u (1-w^2)^{\a-1/2}\, dw,
\end{equation}
which is an odd function in $-1<u<1$. Then $d\Pi_{\a}$ is a probability measure in $[-1,1]$.
As $\a \to -1/2$, one finds that $d\Pi_{\a}$ converges weakly to the measure 
$d\Pi_{-1/2} := \frac{1}{2}(\delta_{-1}+\delta_{1})$, where $\delta_{\pm 1}$
denotes a point mass at $\pm 1$.

Now we observe that \eqref{Pimes} can be extended to all complex $\a \neq -1/2$ with $\Re \a > -1$.
Then the (distribution) derivative 
$$
\pia = \frac{\Gamma(\a+1)}{\sqrt{\pi}\Gamma(\a+1/2)} \, \big(1-u^2\big)^{\a-1/2}\, du
$$
is a local complex measure in $(-1,1)$. For $\a \in (-1,-1/2)$ real, its density 
is negative, even and not integrable in $(-1,1)$. If $\phi$ is a continuous function in $(-1,1)$ and
$\phi(u) = \mathcal{O}(1-u)$ as $u \to 1$, then the integral $I(\a) = \int_0^1\phi(u)\, \pia$ is well
defined. As a function of $\a$, this integral is analytic in $\{\a : \Re \a > -1, \a\neq -1/2\}$.
Since $|I(\a)| \lesssim |\a+1/2|\int_0^1 (1-u^2)^{\Re\a+1/2}\, du \to 0$ as $\a \to -1/2$, we see that
$I(\a)$ is actually analytic in $\{\a : \Re\a > -1\}$ and $I(-1/2)=0$. More generally, if $\phi_{\ab}(u)$
is continuous in $(u,\ab)$ and analytic in $(\ab)$ for $-1<u<1$ and $\Re\a, \Re\b >-1$, and
$\phi_{\ab}(u) = \mathcal{O}(1-u)$ locally uniformly in $(\ab)$, then
$I(\ab) = \int_0^1 \phi_{\ab}(u)\, d\Pi_{\a}(u)$ will be analytic in $(\ab)$ in $\Re\a, \Re\b >-1$. 
Under analogous assumptions, this also extends to functions $\phi_{\ab}(u,v)$ and the double integral
$I(\ab) = \iint_{(0,1)^2} \phi_{\ab}(u,v) \, d\Pi_{\a}(u)\,d\Pi_{\b}(v)$,
if one assumes $\phi_{\ab}(u,v) = \mathcal{O}((1-u)(1-v))$ locally uniformly in $\a$ and $\b$. 

The measures $d\Pi_{\a}$ will now be used to extend the representation \eqref{HI} to the range $\ab > -1$.
Denote
\begin{equation}\label{def:Psi}
\Puv :=   \frac{c_{\ab} \,\sinh\frac{t}2}{(\cosh\frac{t}2-1 + q(\t,\vp,u,v))^{\a+\b+2}}.
\end{equation}
Taking the even parts of $\Puv$ in $u$ and $v$, we also define
$$
\Peeuv := \frac{1}4 \sum_{\xi,\eta = \pm 1} \Psi^{\ab}(t,\t,\vp,\xi u, \eta v).
$$
Notice that by \eqref{HI} and symmetry reasons, for $\ab \ge -1/2$ we have
\begin{align}
\H_t^{\ab}(\t,\vp)  =  4 \iint_{(0,1]^2} \Peeuv \, \pia \, \pib. \label{Hsym}
\end{align}
We can now state a general integral representation of $\H_t^{\ab}(\t,\vp)$.
\begin{thm}\label{thm:JPK}
For all $\a,\b > -1$, $t>0$ and $\t,\vp \in [0,\pi]$,
\begin{align}\label{Hgen}
\H_{t}^{\ab}(\t,\vp)  
& = 
		4 \iint_{ (0,1]^2 } \Big( \Peeuv - \Peeu - \Peev \\ & \qquad \qquad 
		+ \Pee \Big) \, \pia\, \pib \nonumber \\
& \quad
		+ 2 \int_{ (0,1] } \Big( \Peeu -\Pee \Big) \, \pia \nonumber \\\nonumber
& \quad
		+ 2 \int_{ (0,1] } \Big( \Peev -\Pee \Big) \, \pib \\
&		\quad + \Pee. \nonumber
\end{align}
\end{thm}

\begin{proof}
For $\ab\ge -1/2$, \eqref{Hgen} is an easy consequence of \eqref{Hsym}. With 
$\phi_{\ab}(u) = \Peeu -\Pee$, the second integral in \eqref{Hgen} is of the form $I(\ab)$ just described;
observe that $\phi_{\ab}(u) = \mathcal{O}(1-u)$, since the derivative $\partial\Psi_{E}^{\ab}/\partial u$
is bounded locally uniformly in $\a$ and $\b$. The third integral in \eqref{Hgen} is similar. For the
double integral, we let
$$
\phi_{\ab}(u,v) = \Peeuv - \Peeu - \Peev + \Pee,
$$
and get a double integral of type $I(\ab)$. 

The conclusion is that the right-hand side of \eqref{Hgen} is analytic in $(\ab) \in \{z : \Re z > -1\}^2$.
Theorem \ref{thm:JPK} follows, since the left-hand side is also analytic.
\end{proof}

We remark that in Theorem \ref{thm:JPK} it does not matter whether one integrates over the open
interval $(0,1)$ or over $(0,1]$, even when the measure is $d\Pi_{-1/2}$. But in the sequel, it will be
more convenient to use $(0,1]$.

Next we restate the formula of Theorem \ref{thm:JPK} in order to obtain a more suitable representation
of $\H_t^{\ab}(\t,\vp)$ for the kernel estimates in Section \ref{sec:kerest}.
Recall that for $-1<\a<-1/2$, $\Pi_{\a}(u)$ is an odd function, which is negative for $u>0$.
It can easily be verified that the density $|\Pi_{\a}(u)|$ defines a finite measure on $[-1,1]$.
In fact, we have the following.
\begin{lem} \label{lem:ximes}
Let $-1 < \a < -1/2$ be fixed. Then
\begin{equation*}
|\Pi_{\a}(u)| 
\simeq 
|u| (1 - |u|)^{\a+1/2}
\simeq
|u| \frac{d\Pi_{\a+1}(u)}{du}, \qquad u \in (-1,1).
\end{equation*}
\end{lem}

\begin{proof}
These three quantities are even in $u$, and we need consider only
$u\in (0,1)$. It is enough to observe that then $|\Pi_{\a}(u)| \simeq \int_0^u (1-w)^{\alpha-1/2} \, dw$.
\end{proof}

\begin{pro} \label{prop:JP}
Let $t>0$ and $\t,\vp \in [0,\pi]$.
\begin{itemize}
\item[(i)] 
If $\ab \ge -1/2$, then
$$
\H_t^{\ab}(\t,\vp) = \iint \Puv \, \pia \, \pib.
$$
\item[(ii)] 
If $-1<\a<-1/2 \le \b$, then
\begin{align*}
\H_t^{\ab}(\t,\vp) & = \iint \Big\{-\partial_u \Puv \, \oa \, \pib \\
	& \qquad \quad + \Puv \, \piu \, \pib\Big\}.
\end{align*}
\item[(iii)] If $-1 < \b < -1/2 \le \a$, then
\begin{align*}
\H_t^{\ab}(\t,\vp) & = \iint \Big\{-\partial_v \Puv \, \pia \, \Pi_{\b}(v)\, dv \\
	& \qquad \quad + \Puv \, \pia\, \piv\Big\}.
\end{align*}
\item[(iv)] If $-1 < \a, \b < -1/2$, then
\begin{align*}
\H_t^{\ab}(\t,\vp) = & \iint \Big\{\partial_{u} \partial_{v} \Puv \, \oa \, \Pi_{\b}(v)\, dv \\  
& \qquad - \partial_{u}  \Puv \, \oa \, \piv\\ 
& \qquad - \partial_{v}  \Puv \, \piu \, \Pi_{\b}(v)\, dv\\
& \qquad + \Puv \, \piu \, \piv \Big\}.
\end{align*}
\end{itemize}
\end{pro}

\begin{proof}
Item (i) is just \eqref{HI}. To prove the remaining items, we combine Theorem \ref{thm:JPK}, 
Lemma \ref{lem:ximes} and symmetries of the quantity $\Peeuv$, its derivatives in $u$ and $v$, 
and the measures involved. We give further details in case of (ii), leaving similar proofs
of (iii) and (iv) to the reader.

Assume that $-1 < \a < -1/2 \le \b$. Since $d\Pi_{\b}$ is a symmetric probability measure on
$[-1,1]$ and has no atom at $0$, formula \eqref{Hgen} reduces to
\begin{align*}
\H_t^{\ab}(\t,\vp) & = 4 \iint_{(0,1]^2} \Big( \Peeuv - \Peev \Big)\, \pia \, \pib \\
	& \quad + 2 \int_{(0,1]} \Peev \, \pib \\
	& \equiv I_1 + I_2.
\end{align*}
Then, expressing $\Psi^{\ab}_{E}$ via $\Psi^{\ab}$ and making use of the symmetry of $d\Pi_{\b}$,
we see that
\begin{align*}
I_2 & = 4 \iint_{(0,1]^2} \Peeuv \, \piu \, \pib \\
	& = \iint \Puv \, \piu \, \pib.
\end{align*}
In $I_1$ we integrate by parts in the $u$ variable, which is legitimate in view of Lemma \ref{lem:ximes}.
Observe that the integrand in $I_1$ vanishes for $u=1$ and that $\Pi_{\a}(0)=0$. We get
\begin{align*}
I_1
= -4 \iint_{(0,1]^2} \partial_u \Psi_E^{\ab}(t,\t,\vp,u,v) \, \oa \, \pib.
\end{align*}
Inserting the definition of the symmetrization $\Psi_E^{\ab}$, one easily finds that
$$
I_1 = - \iint \partial_u \Psi^{\ab}(t,\t,\vp,u,v) \, \oa \, \pib.
$$
The conclusion follows.
\end{proof}

\begin{rem}
All the representations of $\H_t^{\ab}(\t,\vp)$ contained in Proposition \ref{prop:JP}
are positive in the sense that each of the double integrals (there are one of these in \emph{(i)},
two in \emph{(ii)} and in \emph{(iii)}, and four in \emph{(iv)}) is nonnegative. 
\end{rem}

%%%%%%%%%%%%%%%%%%%%%%%%%%%%%%%%%%%%%%%%%%%%%%%%%%%%%%%%%%%%%%%%%%%%%%%%%%%%%%%%%%%%
\section{Preparatory results} \label{sec:prep}
%%%%%%%%%%%%%%%%%%%%%%%%%%%%%%%%%%%%%%%%%%%%%%%%%%%%%%%%%%%%%%%%%%%%%%%%%%%%%%%%%%%%%
In this section we gather various technical results, altogether forming a transparent
and convenient method of proving standard estimates for kernels defined via the Jacobi-Poisson kernel.
The essence of this technique is a uniform way of handling double integrals against
products of measures of type $d\Pi_{\gamma}$ and $\Pi_{\gamma}(u)\, du$, independently of the integrands.
Then the expressions one has to estimate contain only elementary functions and are relatively simple.

The result below, which is a generalization of \cite[Lemma 4.3]{NoSj}, plays a crucial role in our method 
to prove kernel estimates. It provides a link from estimates emerging from the integral representation of
$\H_t^{\ab}(\t,\vp)$, see Proposition \ref{prop:JP}, to the standard estimates related to the space of
homogeneous type $([0,\pi], d\mu_{\ab},|\cdot|)$.
\begin{lem}\label{lem:bridge}
Let $\ab > -1$. Assume that $\xi_1,\xi_2,\kappa_1,\kappa_2 \ge 0$ are fixed and such that 
$\a+\xi_1+\kappa_1, \, \b+\xi_2+\kappa_2 \ge -1/2$. Then, uniformly in $\t,\vp \in [0,\pi]$, $\t \ne \vp$,  
\begin{align*}
	&\bigg( \st+\svp \bigg)^{2\xi_1} \bigg( \ct + \cvp \bigg)^{2\xi_2} \iint 
\frac{d\Pi_{\a+\xi_1+\kappa_1}(u) \, 
	d\Pi_{\b+\xi_2+\kappa_2}(v) }{( q(\t,\vp,u,v) )^{\a+\b+\xi_1+\xi_2+3\slash 2}}\\
	& \quad
\lesssim \frac{1}{\mu_{\ab}(B(\t,|\t-\vp|))}, \\
	&\bigg( \st+\svp \bigg)^{2\xi_1} \bigg( \ct + \cvp \bigg)^{2\xi_2} \iint 
\frac{d\Pi_{\a+\xi_1+\kappa_1}(u) \, d\Pi_{\b+\xi_2+\kappa_2}(v) }{( q(\t,\vp,u,v) )^{\a+\b+\xi_1+\xi_2+2}}\\
	& \quad
\lesssim \frac{1}{ |\t-\vp| \, \mu_{\ab}(B(\t,|\t-\vp|))}.
\end{align*}
\end{lem}
Note that for any $\ab > -1$ fixed, the $\mu_{\ab}$ measure of the interval $B(\t,|\t-\vp|)$ can be 
described as follows, see \cite[Lemma 4.2]{NoSj},
\begin{equation}\label{ball}
\mu_{\ab}(B(\t,|\t-\vp|)) \simeq |\t-\vp|(\t+\vp)^{2\a + 1} (\pi-\t+\pi-\vp)^{2\b + 1}, 
	\qquad \t,\vp \in [0,\pi].
\end{equation}
Notice also that the right-hand sides of the estimates in Lemma \ref{lem:bridge} are always larger
than a positive constant, uniformly in $\t,\vp \in [0,\pi]$. This fact will be used in the sequel without
further mention.

To prove Lemma \ref{lem:bridge}, we need item (b) in the lemma below. This is a generalization
of \cite[Lemma 4.4]{NoSj} stated in \cite[Lemma 2.3]{NoSz}. 
Item (a) contains estimates obtained in \cite[Lemma 5.3]{NoSj2} and will be used in Section
\ref{sec:JPsharp} below to show sharp estimates of the Jacobi-Poisson kernel.
\begin{lem}
\label{lem:intest}
Let $\kappa \ge 0$ and $\gamma$ and $\nu$ be such that $\gamma > \nu +1/2 \ge 0$. Then
\begin{itemize}
\item[(a)]
$$
\int \frac{d\Pi_{\nu}(s)}{(D-Bs)^{\kappa}(A-Bs)^{\gamma}} \simeq
	\frac{1}{(D-B)^{\kappa} A^{\nu+1/2} (A-B)^{\gamma-\nu-1/2}}, \qquad 0 \le B < A \le D;
$$
\item[(b)]
\[
\int \frac{d\Pi_{\nu+\kappa}(s)}{(A-Bs)^{\gamma}} 
	\lesssim \frac{1}{A^{\nu+1\slash 2}(A-B)^{\gamma -\nu -1/2}}, \qquad 0 \le B < A.
\]
\end{itemize}
\end{lem}

\begin{proof}
Part (a) is proved in \cite{NoSj2}. Part (b) can easily be deduced from (a) since the integral to 
be estimated is controlled by the same integral with $\kappa=0$.
\end{proof}

\begin{proof}[Proof of Lemma \ref{lem:bridge}.]
The reasoning is a combination of the arguments given in the proofs of \cite[Lemma 2.1]{NoSz} and 
\cite[Lemma 4.3]{NoSj}. Since $(\t-\vp)^2 \lesssim q(\t,\vp,u,v)$, it suffices to verify the first estimate.
Further, we may reduce the task to showing that 
\begin{equation}\label{red}
\iint 
\frac{d\Pi_{\a+\kappa_1}(u) \, d\Pi_{\b+\kappa_2}(v) }{( q(\t,\vp,u,v) )^{\a+\b+3\slash 2}}
\lesssim \frac{1}{\mu_{\ab}(B(\t,|\t-\vp|))}, \qquad \t,\vp \in [0,\pi], \quad \t \ne \vp, 
\end{equation}
under the assumption $\a+\kappa_1,\b+\kappa_2 \ge -1/2$. Indeed, applying \eqref{red} 
with $\a+\xi_1,\b+\xi_2$ instead of $\ab$, and then using \eqref{ball}, we obtain
\begin{align*}
	&\bigg( \st+\svp \bigg)^{2\xi_1} \bigg( \ct + \cvp \bigg)^{2\xi_2} \iint
\frac{d\Pi_{\a+\xi_1+\kappa_1}(u) \, 
	d\Pi_{\b+\xi_2+\kappa_2}(v) }{( q(\t,\vp,u,v) )^{\a+\b+\xi_1+\xi_2+3\slash 2}}\\
	& \quad \lesssim
\big( \t+\vp \big)^{2\xi_1} \big( \pi - \t + \pi - \vp \big)^{2\xi_2} 
\frac{1}{\mu_{\a+\xi_1,\b+\xi_2}(B(\t,|\t-\vp|))}
 \simeq
\frac{1}{\mu_{\ab}(B(\t,|\t-\vp|))}.
\end{align*}

To prove \eqref{red}, it is convenient to distinguish two cases.

\noindent \textbf{Case 1:} $\ab \in (-1,-1\slash 2)$.

\noindent Taking into account the estimates, see \cite[(21)]{NoSj},
\begin{equation*}
|\t-\vp|^2 \simeq 2 \sin^2\frac{\t-\vp}{4} 
\le
q(\t,\vp,u,v) 
\le 2 \cos^2\frac{\t-\vp}{4}
\le 2, \qquad \t, \vp \in [0,\pi], \quad u,v \in [-1,1],
\end{equation*}
and the fact that $d\Pi_{\a+\kappa_1}$ and $d\Pi_{\b+\kappa_2}$ are finite, we get
\[
\iint
\frac{d\Pi_{\a+\kappa_1}(u) \, d\Pi_{\b+\kappa_2}(v) }{( q(\t,\vp,u,v) )^{\a+\b+3\slash 2}}
\lesssim
\frac{1}{|\t-\vp|^{2\a+1} |\t-\vp|^{2\b+1} |\t-\vp|} + \chi_{\{\a+\b+3/2 < 0\}}.
\]
Then using the inequalities
$|\t-\vp| \le \t + \vp$ and $|\t-\vp| \le \pi - \t + \pi - \vp$ 
together with \eqref{ball}, we obtain \eqref{red}.

\noindent \textbf{Case 2:} At least one of the parameters $\ab$ is in $[-1\slash 2,\infty)$, 
say $\b \ge -1/2$. 

\noindent Proceeding as in the proof of \cite[Lemma 4.3]{NoSj} but applying Lemma \ref{lem:intest} (b)
instead of \cite[Lemma 4.4]{NoSj} to the integral against $d\Pi_{\b+\kappa_2}$, we see that
\[
\iint 
\frac{d\Pi_{\a+\kappa_1}(u) \, d\Pi_{\b+\kappa_2}(v) }{( q(\t,\vp,u,v) )^{\a+\b+3\slash 2}}
\lesssim
\frac{1}{(\pi-\t+\pi-\vp)^{2\b+1}} \, \int
\frac{d\Pi_{\a+\kappa_1}(u)}{( q(\t,\vp,u,1) )^{\a+1}}.
\]
When $\a \ge -1\slash 2$, another application of Lemma \ref{lem:intest} (b) leads to \eqref{red},
see the proof of \cite[Lemma 4.3]{NoSj}.
If $\a \in (-1,-1/2)$ we can apply the arguments from Case 1 getting
$$
\int \frac{d\Pi_{\a+\kappa_1}(u)}{(q(\t,\vp,u,1))^{\a+1}} \lesssim \frac{1}{|\t-\vp|^{2\a+2}}
	\le \frac{1}{(\t+\vp)^{2\a+1} |\t-\vp|}.
$$
Using now \eqref{ball}, we arrive at the desired conclusion. 

The proof of Lemma \ref{lem:bridge} is complete.
\end{proof}

The remaining part of this section embraces various technical results, which will 
allow us to control the relevant kernels by means of the estimates from Lemma \ref{lem:bridge}.
To state the next lemma and also for further use, we introduce the following notation. We will omit the
arguments and write briefly $\q$ instead of $q(\t,\vp,u,v)$, when it does not lead to confusion.
For a given parameter $\lam \in \mathbb{R}$, we define the auxiliary function
\[
\Plam := \frac{\sinh\frac{t}{2}}{(\cosh\frac{t}{2}-1 + \q)^{\lam}},
\]
so that $\Puv = c_{\a,\b} \Psi^{\a+\b+2}(t,\mathfrak{q})$; see \eqref{def:Psi}.

\begin{lem}\label{lem:EST}
Let $\lam \in \mathbb{R}$, $M,N \in \mathbb{N}=\{0,1,2,\ldots\}$ and $K,R,L \in \{0 , 1\}$ be fixed. Then
\begin{align*}
&\big| 
	\partial_u^K \partial_v^R \partial_\vp^L \partial_\t^N \partial_t^M \Plam
\big| \\
&\quad
\lesssim 
\sum_{k,r=0,1,2} 
\bigg( \st+\svp \bigg)^{Kk} \bigg( \ct + \cvp \bigg)^{Rr} 
\frac{1}{(t^2 + \q)^{   \lam + (L+N+M-1 + Kk +Rr)\slash 2  }},
\end{align*}
uniformly in $t\in (0,1]$, $\t,\vp \in [0,\pi]$ and $u,v \in [-1,1]$.
\end{lem}

To prove this lemma, we need two preparatory results. One of them is 
Fa\`a di Bruno's formula for the $N$th derivative, $N \ge 1$, of the composition
of two functions (see \cite{Jo} for the related references and interesting historical remarks). 
With $D$ denoting the ordinary derivative, it reads
\begin{equation} \label{Faa}
D^N(g\circ f)(\t) = \sum \frac{N!}{j_1! \cdot \ldots \cdot j_N!} \;
	\big(D^{j_1+\ldots+j_N} g\big)
	\circ f(\t) \bigg( \frac{D^1 f(\t)}{1!}\bigg)^{j_1}\cdot \ldots \cdot
	\bigg( \frac{D^N f(\t)}{N!}\bigg)^{j_N},
\end{equation}
where the summation runs over all $j_1,\ldots,j_N \ge 0$ such that $j_1+2j_2+\ldots+N j_N = N$.
Further, in the proof of Lemma \ref{lem:EST} we will make use of
the following bounds given in \cite{NoSj}.

\begin{lem}[{\cite[Lemma 4.5]{NoSj}}] \label{lem:NoSj4.5}
For all $\t,\vp \in [0,\pi]$ and $u,v \in [-1,1]$, one has
\[
\big|
\partial_\t \q
\big|
\lesssim \sqrt{\q} \qquad \text{and} \qquad 
\big|
\partial_\vp \q
\big|
\lesssim \sqrt{\q}.
\]
\end{lem}

\begin{proof}[Proof of Lemma \ref{lem:EST}.]  
Given $\lam \in \mathbb{R}$, we introduce the auxiliary function
\[
\Ptilde := \frac{1}{\sinh\frac{t}2} \, \Plam = \frac{1}{(\cosh\frac{t}2 - 1 + \q)^{\lam}}.
\]
We first reduce our task to showing the estimate
\begin{align}
&\big| 
	\partial_u^K \partial_v^R \partial_\vp^L \partial_\t^N  \Ptilde
\big| \nonumber\\
&\quad
\lesssim 
\sum_{k,r=0,1,2} 
\bigg( \st+\svp \bigg)^{Kk} \bigg( \ct + \cvp \bigg)^{Rr} 
\frac{1}{(t^2 + \q)^{   \lam + (L+N + Kk +Rr)\slash 2  }} \label{ESTred}
\end{align}
for $t \in (0,1]$, $\t,\vp \in [0,\pi]$ and $u,v\in [-1,1]$; 
here $\lam \in \mathbb{R}$, $N \in \mathbb{N}$ and $K,R,L \in \{0 , 1\}$ are fixed.

Observe that
\[
\Plam
=c_\lam
\left\{ \begin{array}{ll}
\partial_t \frac{1}{(\cosh\frac{t}{2} - 1 + \q)^{\lam - 1}}
, & \lam \ne 1\\
\partial_t \log (\cosh\frac{t}{2} - 1 + \q), & \lam = 1
\end{array} \right.,
\]
where $c_\lam$ is a constant, possibly negative.
Using Fa\`a di Bruno's formula \eqref{Faa} with $f(t)=\cosh\frac{t}2 - 1 + \q$ and either 
$g(x)=x^{-\lam+1}$ or $g(x)=\log x$, we obtain
\begin{align*} 
\partial_t^M \Plam 
&= 
c_{\lam} \, \partial_t^{M+1} (g\circ f) (t) \\
&=
\sum_{\substack{j_i \ge 0 \\ j_1+\ldots+ (M+1)j_{M+1}=M+1}} C_{\lam,j} 
\Big(\sinh\frac{t}{2}\Big)^{\oddM} \Big(\cosh\frac{t}{2}\Big)^{\evenM}
\Ptildej,
\end{align*}
where the $C_{\lam,j}$ are constants, possibly zero.
Differentiating these identities with respect to 
$\t,\vp,u,v$ and then applying \eqref{ESTred} and the relations
\begin{equation*}
\cosh\frac{t}2 \simeq 1, \qquad
\sinh\frac{t}{2} \simeq t \le \sqrt{t^2 +\q},
\qquad \quad t \in (0,1],
\end{equation*}
we see that
\begin{align*}
\big| 
	\partial_u^K \partial_v^R \partial_\vp^L \partial_\t^N \partial_t^M \Plam
\big| 
	&\lesssim
\sum_{\substack{j_i \ge 0 \\ j_1+\ldots+ (M+1)j_{M+1}=M+1}}
\sum_{k,r=0,1,2} 
\bigg( \st+\svp \bigg)^{Kk} \bigg( \ct + \cvp \bigg)^{Rr} \\
&\quad \qquad \times
\frac{1}{(t^2 + \q)^{   \lam -1 + \sum_i j_i - (\oddM) \slash 2+ (L+N + Kk +Rr)\slash 2  }}.
\end{align*}
Now by the boundedness of $\q$ and the inequality
\begin{equation}\label{estj}
\sum_i j_i - \frac{1}{2} \oddM 
\le \frac{M+1}2,
\end{equation}
forced by the constraint $j_1+\ldots+ (M+1)j_{M+1}=M+1$,
we get the asserted estimate. Thus it remains to prove \eqref{ESTred}.

We assume that $N\ge 1$. The simpler case $N = 0$ is left to the reader.
Taking into account the relations
\[
\partial_\t^{2m} \q = (-4)^{-m}(\q-1), 
\quad
\partial_\t^{2m-1} \q = (-4)^{1-m} \partial_\t \q,
\qquad m\ge 1,
\]
see \cite[Section 4]{NoSj},
and using Fa\`a di Bruno's formula with $f(\t) = \cosh\frac{t}{2} - 1 + \q$ and $g(x) = x^{-\lam}$, we get 
\begin{align*}
\partial_\t^N \Ptilde =
 \sum_{\substack{j_i \ge 0 \\ \,\quad j_1+\ldots+ Nj_{N}=N}} c_{\lam,j} 
\frac{1}{(\cosh\frac{t}{2} - 1 + \q)^{\lam + \sum_i j_i}} 
\,
(\q-1)^{\evenN} (\partial_\t \q)^{\oddN},
\end{align*}
where $c_{\lam,j}$ are constants.
Further, keeping in mind that $L,R,K \in \{ 0,1 \}$ and applying repeatedly Leibniz' rule, we see
that $\partial_\vp^L \partial_\t^N  \Ptilde$ is a sum of terms of the form constant times 
\begin{align*}
\frac{1}{(\cosh\frac{t}{2} - 1 + \q)^{\lam + \sum_i j_i + l_1}} 
\,
(\partial_\vp \q)^{l_1+l_2}
(\q-1)^{\evenN - l_2} (\partial_\t \q)^{\oddN- l_3}
(\partial_\vp \partial_\t \q)^{l_3},
\end{align*}
where the indices run over the set described by the conditions $j_i \ge 0$, $j_1+\ldots+ Nj_{N}=N$, 
$l_1, l_2, l_3 \ge 0$, $l_1+l_2+l_3=L$ and the exponents of $\q - 1$ and $\partial_\t \q$ are nonnegative.
Similarly, $\partial_v^R \partial_\vp^L \partial_\t^N  \Ptilde$ is a sum of terms of the form constant
times 
\begin{align*}
&\frac{1}{(\cosh\frac{t}{2} - 1 + \q)^{\lam + \sum_i j_i + l_1+r_1}}
(\partial_v \q)^{r_1+r_3}
(\partial_\vp \q)^{l_1+l_2-r_2} 
(\partial_v \partial_\vp \q)^{r_2} 
(\q-1)^{\evenN - l_2 - r_3} \\
	&\qquad \times
(\partial_\t \q)^{\oddN- l_3 - r_4}
(\partial_v \partial_\t \q)^{r_4}
(\partial_\vp \partial_\t \q)^{l_3 - r_5}
(\partial_v \partial_\vp \partial_\t \q)^{r_5},
\end{align*}
where also 
$r_1, \ldots, r_5 \ge 0$, $r_1+\ldots+r_5=R$, $l_1 + l_2 \ge r_2$, $l_3 \ge r_5$.
Finally, since the derivative $\partial_u \partial_v \q$ vanishes,
$\partial_u^K \partial_v^R \partial_\vp^L \partial_\t^N  \Ptilde$ is a sum of terms of the form 
constant times 
\begin{align*}
&\frac{1}{(\cosh\frac{t}{2} - 1 + \q)^{\lam + \sum_i j_i + l_1+r_1+k_1}}
(\partial_u \q)^{k_1+k_3}  (\partial_v \q)^{r_1+r_3} 
(\partial_\vp \q)^{l_1+l_2-r_2-k_2} (\partial_u \partial_\vp \q)^{k_2}
(\partial_v \partial_\vp \q)^{r_2}   \\
	&\qquad \times
(\q-1)^{\evenN - l_2 - r_3 - k_3} (\partial_\t \q)^{\oddN- l_3 - r_4 - k_4}
(\partial_u \partial_\t \q)^{k_4}
(\partial_v \partial_\t \q)^{r_4} 
(\partial_\vp \partial_\t \q)^{l_3 - r_5 - k_5} \\
	&\qquad \times
(\partial_u \partial_\vp \partial_\t \q)^{k_5}
(\partial_v \partial_\vp \partial_\t \q)^{r_5}.
\end{align*}
Here we must add the conditions $k_1, \ldots, k_5 \ge 0$, $k_1+\ldots+k_5=K$
and replace $l_1 + l_2 \ge r_2$, $l_3 \ge r_5$ by $l_1 + l_2 \ge r_2 + k_2$, $l_3 \ge r_5 + k_5$.
We shall estimate all the factors in this product from above. Since $t \le 1$, 
we can replace $\cosh\frac{t}2-1+\q$ by $t^2+\q$.
The quantities $\q$ and $\partial_\vp \partial_\t \q$ are bounded. Further, we apply 
Lemma \ref{lem:NoSj4.5} to get
\[
|\partial_\vp \q| + |\partial_\t \q| \lesssim \sqrt\q \le \sqrt{t^2+\q}.
\]
To deal with the resulting exponent of $1/(t^2 + \q)$, we observe that
\[
l_1-l_2+l_3 \le L, \qquad
\sum_i j_i - \frac{1}{2}\oddN \le \frac{N}2,
\]
cf{.} \eqref{estj}. Using also the estimates
\begin{displaymath}
\begin {array}{lll}
&|\partial_u \q| \le \big( \st+\svp \big)^{2}, \qquad &
|\partial_v \q| \le \big( \ct + \cvp \big)^{2}, \\ 
&|\partial_\t \partial_u \q| + |\partial_\vp \partial_u \q| \le \st+\svp, \qquad &
|\partial_\t \partial_v \q| + |\partial_\vp \partial_v \q| \le \ct + \cvp, \\
&|\partial_\vp \partial_\t \partial_u \q| \le 1, \qquad & 
|\partial_\vp \partial_\t \partial_v \q| \le 1,
\end {array}
\end{displaymath}
we infer that
\begin{align*}
\big| \partial_u^K \partial_v^R \partial_\vp^L \partial_\t^N  \Ptilde \big|
	\lesssim & 
\sum_{\substack{r_1+\ldots+r_5=R \\ k_1+\ldots+k_5=K }} 
\bigg( \st+\svp \bigg)^{2k_1+2k_3+k_2+k_4} 
\bigg( \ct+\cvp \bigg)^{2r_1+2r_3+r_2+r_4} \\
	& \quad \times
\frac{1}{(t^2 + \q)^{\lam + (N+L+ 2k_1+k_2+k_4 + 2r_1+r_2+r_4)\slash 2}} .
\end{align*}
Notice that $2k_1+k_2+k_4 \in \{0,K,2K \}$, and similarly 
$2r_1+r_2+r_4 \in \{0,R,2R \}$. 
This observation leads directly to \eqref{ESTred}.
The proof of Lemma \ref{lem:EST} is complete.
\end{proof}
Define
$$
d\abpiK = \begin{cases}
		d\Pi_{-1\slash 2}, & K=0 \\
		d\Pi_{\a + 1}, &  K=1
	\end{cases},
$$
and similary for $d\abpiR$.
\begin{cor} \label{cor:t1}
Let $M,N \in \mathbb{N}$ and $L \in \{0 , 1\}$ be fixed. 
The following estimates hold uniformly in $t\in (0,1]$ and $\t,\vp \in [0,\pi]$.
\begin{itemize}
\item[(i)]
If $\a,\b \ge -1\slash 2$, then 
\begin{align*}
\big| 
	 \partial_\vp^L \partial_\t^N \partial_t^M 
H_{t}^{\ab}(\t,\vp)  
\big|  
	\lesssim 
\iint
\frac{\pia \, \pib}{(t^2 + \q)^{   \a + \b + 3\slash 2 + (L+N+M)\slash 2  }}.
\end{align*}
\item[(ii)]
If $-1 < \a < -1\slash 2 \le \b$, then 
\begin{align*}
\big| 
	 \partial_\vp^L \partial_\t^N \partial_t^M 
H_{t}^{\ab}(\t,\vp)  
\big|  
	&\lesssim 
1 + 
\sum_{K=0,1}
\sum_{k=0,1,2} 
\bigg( \st+\svp \bigg)^{Kk}  \\ 
& \qquad \qquad \times
\iint
\frac{\piK \, \pib}{(t^2 + \q)^{   \a + \b + 3\slash 2 + (L+N+M + Kk)\slash 2  }}.
\end{align*}
\item[(iii)]
If $-1 < \b < -1\slash 2 \le \a$, then 
\begin{align*}
\big| 
	\partial_\vp^L \partial_\t^N \partial_t^M 
H_{t}^{\ab}(\t,\vp) 
\big| 
	&\lesssim 
1 + 
\sum_{R=0,1}
\sum_{r=0,1,2} 
\bigg( \ct + \cvp \bigg)^{Rr} \\
& \qquad \qquad \times
\iint
\frac{\pia \, \piR}{(t^2 + \q)^{   \a + \b + 3\slash 2 + (L+N+M +Rr)\slash 2  }}.
\end{align*}
\item[(iv)] 
If $-1 < \a,\b < -1\slash 2$, then
\begin{align*}
\big| 
	\partial_\vp^L \partial_\t^N \partial_t^M 
H_{t}^{\ab}(\t,\vp) 
\big| 
	&\lesssim 
1 + 
\sum_{K,R=0,1}
\sum_{k,r=0,1,2} 
\bigg( \st+\svp \bigg)^{Kk} \bigg( \ct + \cvp \bigg)^{Rr} 
\\ 
& \qquad \qquad \times
\iint
\frac{\piK \, \piR}{(t^2 + \q)^{   \a + \b + 3\slash 2 + (L+N+M + Kk +Rr)\slash 2  }}.
\end{align*}
\end{itemize}
\end{cor}

\begin{proof}
All the bounds are direct consequences of the equality \eqref{JPH}, Proposition \ref{prop:JP}, 
Lemma \ref{lem:ximes} and the estimate from Lemma \ref{lem:EST} (specified to $\lam = \a + \b + 2$).
Here passing with the differentiation in $t$, $\t$ or $\vp$ under integrals against 
$d\Pi_{\gamma}$, $\gamma \ge -1/2$, or $\Pi_{\gamma}(u)\, du$, $-1 < \gamma < -1/2$,
is legitimate and can easily be justified with the aid of Lemma \ref{lem:EST}
and the dominated convergence theorem.
\end{proof}

\begin{lem}\label{lem:integral}
Let $\gamma \in \mathbb{R}$ and $\eta \ge 0$ be fixed. Then
\[
\int_0^1 \frac{t^\eta \,dt}{(t^2+ \rho )^{\gamma + \eta\slash 2 +  1\slash 2}}
\lesssim
\left\{ \begin{array}{lll}
\rho^{-\gamma}
, & \gamma > 0\\
\log \big(1+\rho^{-1/2} \big), & \gamma = 0 \\
1, & \gamma < 0
\end{array} \right.,
\]
uniformly in $0 < \rho \le 2$. 
\end{lem}

\begin{proof}
This is elementary. For $\gamma=0$, one has
\begin{equation*}
\int_0^1 \frac{t^{\eta} \, dt}{(t^2+ \rho )^{\eta/2 + 1\slash 2}} 
\le
\int_0^1 \frac{dt}{(t^2+ \rho )^{1\slash 2}} 
\simeq
\int_0^1 \frac{dt}{t+ \rho^{1/2}}
=
\log \Big( 1 + \rho^{-1/2} \Big).
\end{equation*}
\end{proof}

The next lemma will be frequently used in Section \ref{sec:kerest} to prove the relevant kernel estimates.
Only the cases $p\in \{ 1,2,\infty \}$ will be needed for our purposes.
Other values of $p$ are also of interest, but in connection with operators not considered in this paper.

\begin{lem}\label{lem:finbridge}
Let $K,R \in \{ 0,1 \}$, $k,r \in \{ 0,1,2 \}$, $W \ge 1$, $s\ge 0$ and $1 \le p \le \infty$ be fixed.
Consider a function $\Upsilon^{\ab}_s(t,\t,\vp)$ defined on $(0,1) \times [0,\pi] \times [0,\pi]$ 
in the following way.
\begin{itemize}
\item[(i)]
For $\a,\b \ge -1\slash 2$,
\begin{align*}
\Upsilon^{\ab}_s(t,\t,\vp) := 
\iint
\frac{\pia \, \pib }{(t^2 + \q)^{ \a+\b+3\slash 2 + W\slash (2p) + s\slash 2  }}.
\end{align*}
\item[(ii)]
For $-1 < \a < -1\slash 2 \le \b$,
\begin{align*}
\Upsilon^{\ab}_s(t,\t,\vp) := 
\bigg( \st + \svp \bigg)^{Kk} \iint
\frac{\piK \, \pib }{(t^2 + \q)^{ \a+\b+3\slash 2 + W\slash (2p) +Kk\slash 2 + s\slash 2  }}.
\end{align*}
\item[(iii)]
For $-1 < \b < -1\slash 2 \le \a$,
\begin{align*}
\Upsilon^{\ab}_s(t,\t,\vp) := 
\bigg( \ct + \cvp \bigg)^{Rr} \iint
\frac{\pia \, \piR }{(t^2 + \q)^{ \a+\b+3\slash 2 + W\slash (2p) +Rr\slash 2 + s\slash 2  }}.
\end{align*}
\item[(iv)] 
For $-1 < \a,\b < -1\slash 2$,
\begin{align*}
\Upsilon^{\ab}_s(t,\t,\vp) := &
\bigg( \st+\svp \bigg)^{Kk} \bigg( \ct + \cvp \bigg)^{Rr} \\
& \quad \times \iint
\frac{\piK \, \piR }{(t^2 + \q)^{ \a+\b+3\slash 2 + W\slash (2p) +Kk\slash 2 + Rr\slash 2 + s\slash 2  }}.
\end{align*}
\end{itemize}
Then the estimate
\[
\| 1 + \Upsilon^{\ab}_s(t,\t,\vp) \|_{ L^p((0,1),t^{W-1}dt) } 
\lesssim
\frac{1}{|\t-\vp|^s} \; \frac{1}{\mu_{\ab}(B(\t,|\t-\vp|))}
\]
holds uniformly in $\t,\vp \in [0,\pi]$, $\t\ne\vp$.
\end{lem}

\begin{proof}
It is enough to prove the desired estimate without the term $1$ in the left-hand side.
Further, since $|\t-\vp|^2 \lesssim \q$, it suffices to consider the case $s=0$. 
We prove the estimate when $-1 <\ab<-1 \slash 2$. The remaining cases  
are left to the reader; they are simpler since then  $\a+\b+3\slash 2 > 0$ and one needs 
Lemma \ref{lem:integral} only with $\gamma > 0$.

We first assume that $p<\infty$. Using Minkowski's integral inequality and then Lemma \ref{lem:integral} 
with $\gamma = p(\a+\b+3\slash 2 + Kk\slash 2 + Rr\slash 2)$, $\eta= W-1$ and $\rho=\q$, we obtain
\begin{align*}
	&\| \Upsilon^{\ab}_0(t,\t,\vp) \|_{ L^p((0,1),t^{W-1}dt) } \\
	&\quad \le
\bigg( \st+\svp \bigg)^{Kk} \bigg( \ct + \cvp \bigg)^{Rr} 
 \iint \Bigg(  \int_0^1
\frac{t^{W-1} \, dt }{(t^2 + \q)^{ p(\a+\b+3\slash 2 + W\slash (2p) +Kk\slash 2 + Rr\slash 2)  }} 
\Bigg)^{1\slash p}\\
	& \qquad \qquad \qquad  \times \piK \, \piR \\
	&\quad \lesssim
\bigg( \st+\svp \bigg)^{Kk} \bigg( \ct + \cvp \bigg)^{Rr} 
 \iint \bigg[
\bigg( \frac{1}{\q}  \bigg)^{\a+\b+3\slash 2 + Kk\slash 2 + Rr\slash 2}+1 \\
	& \qquad \qquad \qquad 
 + \Big( \log \big( 1 + {\q}^{-1/2} \big) \Big)^{1\slash p}
 \bigg]\,
\piK \, \piR.
\end{align*}
Now an application of Lemma \ref{lem:bridge} 
(specified to $\xi_1=Kk\slash 2$, 
$\kappa_1=-\a-1\slash 2$ if $K=0$ and $\kappa_1=1 - k\slash 2$ if $K=1$, $\xi_2=Rr\slash 2$, $\kappa_2=-\b-1\slash 2$ if $R=0$ and $\kappa_2=1 - r\slash 2$ if $R=1$) 
gives the desired estimate for the expression emerging from the first term in the last integral.
As for the remaining two expressions, we observe that 
$1 \lesssim \log \big( 1 + \q^{-1/2}\big) \lesssim \log \big( 1 + |\t-\vp|^{-1} \big)$. 
Moreover, as can be seen from \eqref{ball}, there exists an $\varepsilon = \varepsilon(\ab) > 0$ such that
\[
\mu_{\ab}\big( B(\t,|\t - \vp|) \big) 
	\lesssim |\t - \vp|^\varepsilon, \qquad \t,\vp \in [0,\pi]. 
\]
Since the measures $d\abpiK$ and $d\abpiR$ are finite, the conclusion follows.

The case $p=\infty$ can be justified in a similar way by using in the reasoning above the estimate
\[
\frac{1}{ (t^2 + \q)^{ \a+\b+3\slash 2 +Kk\slash 2 + Rr\slash 2} }
\lesssim
\bigg(\frac{1}{\q}\bigg)^{ \a+\b+3\slash 2 +Kk\slash 2 + Rr\slash 2}  +1, \qquad t \in (0,1),
\]
instead of Lemma \ref{lem:integral}.
\end{proof}

The next lemma and corollaries are long-time counterparts of Corollary \ref{cor:t1} 
and Lemma \ref{lem:finbridge}.
\begin{lem} \label{lem:tL}
Assume that $M,N \in \N$ and $L \in \{0,1\}$ are fixed. Given $\ab > -1$, there exists an
$\epsilon = \epsilon(\ab)>0$ such that
\begin{align*}
& \big| \partial_{\vp}^{L} \partial_{\t}^N \partial_t^M H_t^{\ab}(\t,\vp)\big| \\ & \quad \lesssim
e^{- t \left( \left| \frac{\a + \b + 1}{2}  \right| + \epsilon \right)} 
+ \chi_{\{N=L=0, \, \a+\b+1 \ne 0\}} e^{- t  \left| \frac{\a + \b + 1}{2} \right|}
+ \chi_{\{M=N=L=0, \, \a+\b+1=0\}},
\end{align*}
uniformly in $t \ge 1$ and $\t,\vp \in [0,\pi]$. Moreover, one can take 
$\epsilon = (\a+\b+2) \wedge 1$.
\end{lem}
To prove this, it is more convenient to employ the series representation of $H_t^{\ab}(\t,\vp)$
rather than the formulas from Proposition \ref{prop:JP}.
\begin{proof}[{Proof of Lemma \ref{lem:tL}}]
For $\ab>-1$, $t>0$ and $\t,\vp \in [0,\pi]$ we have
\begin{equation} \label{LTdec}
H_t^{\ab}(\t,\vp) = \frac{1}{\mu_{\ab}([0,\pi])} e^{-t\left|\frac{\a+\b+1}{2}\right|}
+ \sum_{n=1}^{\infty} e^{-t\left(n+\frac{\a+\b+1}2\right)} \P_n^{\ab}(\t) \P_n^{\ab}(\vp).
\end{equation}
Denote the sum in \eqref{LTdec} by $S$.
To estimate $S$ and its derivatives, we will need suitable bounds for
$\partial_{\t}^{N}\P_n^{\ab}(\t)$, $N \ge 0$. It is known (see \cite[(7.32.2)]{Sz}) that
\begin{equation} \label{Jpolb}
|\P_n^{\ab}(\t)| \lesssim n^{\a+\b+2}, \qquad \t \in [0,\pi], \quad n \ge 1.
\end{equation}
Combining this with the identity (cf. \cite[(4.21.7)]{Sz})
$$
\partial_{\t} \P_n^{\ab}(\t) = -\frac{1}{2} \sqrt{n(n+\a+\b+1)} \sin\t \, \P_{n-1}^{\a+1,\b+1}(\t),
\qquad n \ge 1,
$$
we see that for each $N \ge 0$
$$
|\partial_{\t}^N \P_n^{\ab}(\t)| \lesssim n^{3N+\a+\b+2}, \qquad \t \in [0,\pi], \quad n \ge 1.
$$

In view of these facts, the series in \eqref{LTdec} can be repeatedly differentiated 
term by term in $t,\t$ and $\vp$, and we get the bounds
\begin{align*}
\partial_{\vp}^{L} \partial_{\t}^N \partial_t^M S & \lesssim 
\sum_{n=1}^{\infty} e^{-t\left(n+\frac{\a+\b+1}2\right)} n^{M+3N+3L+2\a+2\b+4} \\
& = e^{-t\left( \left|\frac{\a+\b+1}2\right| + (\a+\b+2)\wedge 1 \right)} \sum_{n=1}^{\infty} 
	e^{-t\left(n - 1\right)} n^{M+3N+3L+2\a+2\b+4} \\
& \lesssim e^{-t\left( \left|\frac{\a+\b+1}2\right| + (\a+\b+2)\wedge 1 \right)},
\end{align*}
uniformly in $t \ge 1$ and $\t,\vp \in [0,\pi]$. 

Since the other term in \eqref{LTdec} is trivial to handle, the conclusion follows.
\end{proof}

\begin{cor} \label{cor:tL}
Let $\ab > -1$, $M,N \in \N$, $L \in \{0,1\}$, $W \ge 1$ and $1 \le p \le \infty$ be fixed.
Then
$$
\bigg\|\sup_{\t,\vp \in [0,\pi]} \big|\partial_{\vp}^{L} \partial_{\t}^N \partial_t^M H_t^{\ab}(\t,\vp) \big|
	\bigg\|_{L^p((1,\infty),t^{W-1}dt)} < \infty,
$$
excluding the cases when simultaneously $\a+\b+1=0$ and $M=N=L=0$ and $p<\infty$.
\end{cor}

A strengthened special case of Corollary \ref{cor:tL} will be needed when we estimate kernels
associated with multipliers of Laplace-Stieltjes type.
\begin{cor}\label{cor:tLStieltjes}
Let $\ab > -1$ and 
$L,N \in \{0,1\}$ be fixed.
Then
$$
\bigg\| e^{t \left| \frac{\a + \b + 1}{2} \right|} \sup_{\t,\vp \in [0,\pi]} 
	\big|\partial_{\vp}^{L} \partial_{\t}^N H_t^{\ab}(\t,\vp) \big|
	\bigg\|_{L^\infty((1,\infty),dt)} < \infty.
$$
\end{cor}

%%%%%%%%%%%%%%%%%%%%%%%%%%%%%%%%%%%%%%%%%%%%%%%%%%%%%%%%%%%%%%%%%%%%%%%%%%%%%%%%%%%%
\section{Kernel estimates} \label{sec:kerest}
%%%%%%%%%%%%%%%%%%%%%%%%%%%%%%%%%%%%%%%%%%%%%%%%%%%%%%%%%%%%%%%%%%%%%%%%%%%%%%%%%%%%%

Let $\mathbb{B}$ be a Banach space and let $K(\t,\vp)$ be a kernel defined on 
$[0,\pi]\times[0,\pi]\backslash \{ (\t,\vp):\t=\vp \}$ and taking values in $\mathbb{B}$.
We say that $K(\t,\vp)$ is a \emph{standard kernel} in the sense of the space of homogeneous type
$([0,\pi], d\mu_{\ab},|\cdot|)$ if it satisfies the so-called \emph{standard estimates}, i{.}e{.},\
the growth estimate
\begin{equation} \label{gr}
\|K(\t,\vp)\|_{\mathbb{B}} \lesssim \frac{1}{\mu_{\ab}(B(\t,|\t-\vp|))}
\end{equation}
and the smoothness estimates
\begin{align}
\| K(\t,\vp)-K(\t',\vp)\|_{\mathbb{B}} & \lesssim \frac{|\t-\t'|}{|\t-\vp|}\,
 \frac{1}{\mu_{\ab}(B(\t,|\t-\vp|))},
\qquad |\t-\vp|>2|\t-\t'|, \label{sm1}\\
\| K(\t,\vp)-K(\t,\vp')\|_{\mathbb{B}} & \lesssim \frac{|\vp-\vp'|}{|\t-\vp|}\,
 \frac{1}{\mu_{\ab}(B(\t,|\t-\vp|))},
\qquad |\t-\vp|>2|\vp-\vp'| \label{sm2}.
\end{align}
Notice that in these formulas, the ball (interval) $B(\t,|\t-\vp|)$ can be replaced by $B(\vp,|\vp-\t|)$, 
in view of the doubling property of $\mu_{\ab}$.

We will show that the following kernels, with values in properly chosen Banach spaces $\mathbb{B}$, 
satisfy the standard estimates.
\begin{itemize}
\item[(I)] The kernel associated with the Jacobi-Poisson semigroup maximal operator,
$$
\mathfrak{H}^{\ab}(\t,\vp) = \big\{H_t^{\ab}(\t,\vp)\big\}_{t>0}, 
	\qquad \mathbb{B}=\mathbb{X} \subset L^{\infty}(dt),
$$
where $\mathbb{X}$ is the closed separable subspace
of $L^{\infty}(dt)$ consisting of all continuous functions $f$ on $(0,\infty)$ which have finite
limits as $t \to 0^+$ and as $t \to \infty$. Observe that $\big\{H_t^{\ab}(\t,\vp)\big\}_{t>0} 
\in \mathbb{X}$, for $\t \ne \vp$, as can be seen from Proposition \ref{prop:JP} and the bound 
$\q \gtrsim (\t-\vp)^2$, and the series representation (see the proof of Lemma \ref{lem:tL}).
\item[(II)] The kernels associated with Riesz-Jacobi transforms,
$$
R_N^{\ab}(\t,\vp) = \frac{1}{\Gamma(N)} \int_0^{\infty} \partial_\t^N H_t^{\ab}(\t,\vp)
	t^{N -1}\, dt, \qquad \mathbb{B}=\mathbb{C},
$$
where $N = 1,2,\ldots$. 
\item[(III)] The kernels associated with mixed square functions,
$$
\mathfrak{G}^{\ab}_{M,N}(\t,\vp) = \big\{\partial_\t^N \partial_t^M H_t^{\ab}(\t,\vp) \big\}_{t>0}, \qquad
	\mathbb{B} = L^2(t^{2M+2N-1}dt),
$$
where $M,N = 0,1,2,\ldots$ are such that $M+N>0$.
\item[(IVa)] The kernels associated with Laplace transform type multipliers,
$$
K^{\ab}_{\phi}(\t,\vp) = - \int_0^{\infty} \phi(t) \, \partial_t H_t^{\ab}(\t,\vp)
\, dt, \qquad
	\mathbb{B}=\mathbb{C},
$$
where $\phi \in L^{\infty}(dt)$.
\item[(IVb)] The kernels associated with Laplace-Stieltjes transform type multipliers,
$$
K^{\ab}_{\nu}(\t,\vp) =  \int_{(0,\infty)} H_t^{\ab}(\t,\vp)\, d\nu(t), \qquad
	\mathbb{B}=\mathbb{C},
$$
where $\nu$ is a signed or complex Borel measure on $(0,\infty)$ 
with total variation $|\nu|$ satisfying 
\begin{equation}\label{assum}
\int_{(0,\infty)} e^{-t \left| \frac{\a + \b + 1}{2} \right| } \, d|\nu|(t) < \infty.
\end{equation}
\end{itemize}

When $K(\t,\vp)$ is scalar-valued, i.e.\ $\mathbb{B}=\mathbb{C}$, it is well known that the bounds \eqref{sm1}
and \eqref{sm2} follow from the more convenient gradient estimate
\begin{equation} \label{grad}
{\|\partial_{\t} K(\t,\vp)\|}_{\mathbb{B}} + {\|\partial_{\vp} K(\t,\vp)\|}_{\mathbb{B}}
\lesssim \frac{1}{|\t-\vp|\mu_{\ab}(B(\t,|\t-\vp|))}.
\end{equation}
We shall see that the same holds also in the vector-valued cases we consider. Then the derivatives in
\eqref{grad} are taken in the weak sense, which means that for any $\texttt{v}\in \mathbb{B}^*$
\begin{equation}\label{wder}
\big \langle \texttt{v}, \partial_{\t} K(\t,\vp) \big\rangle =
\partial_{\t} \big \langle \texttt{v}, K(\t,\vp) \big\rangle
\end{equation}
and similarly for $\partial_{\vp}$.
If these weak derivatives $\partial_{\t} K(\t,\vp)$ and $\partial_{\vp} K(\t,\vp)$ exist 
as elements of $\mathbb{B}$ and their norms satisfy \eqref{grad}, the scalar-valued case applies 
and \eqref{sm1} and \eqref{sm2} follow.

The result below extends to all $\ab > -1$ the estimates obtained in
\cite[Section 4]{NoSj} for the restricted range $\ab \ge -1\slash 2$. Moreover, here we also consider 
multipliers of Laplace and Laplace-Stieltjes transform type, which were merely mentioned in \cite{NoSj} and 
which cover as a special case the imaginary powers of $\mathcal{J}^{\ab}$ 
(or $\J^{\ab}\Pi_0$ when $\a+\b+1=0$) investigated there.
\begin{thm} \label{thm:std}
Let $\ab > -1$. Then the kernels \emph{(I)-(III)}, \emph{(IVa)} and \emph{(IVb)} 
satisfy the standard estimates \eqref{gr},
\eqref{sm1} and \eqref{sm2} with $\mathbb{B}$ as indicated above.
\end{thm}

In the proof we tacitly assume that passing
with the differentiation in $\t$ or $\vp$ under integrals against $dt$ or $d\nu(t)$
is legitimate. In fact, such manipulations can easily be verified by means of the dominated convergence 
theorem and the estimates obtained in Corollary \ref{cor:t1} and Lemma \ref{lem:tL}.

%%%%%%%%%%%%%%%%%%%%%%%%%%%%%%%%%%%%%%%%%%%%%%%%%%%%%%%%%%%%%%%%%%%%%%%%%%%%%%%%%%%%%%%%%%%%%%
\begin{proof}[Proof of Theorem \ref{thm:std}] We treat each of the kernels separately.

\noindent
{\bf The case of $\boldsymbol{\mathfrak{H}^{\ab}(\t,\vp)}$.}
We first deal with the growth condition. Clearly, it suffices to prove 
independently the two bounds emerging from \eqref{gr} by choosing 
$\mathbb{B} = L^{\infty} ( (1,\infty), dt )$ and $\mathbb{B} = L^{\infty} ( (0,1), dt )$.
These, however, are immediate consequences of Corollary \ref{cor:tL} (with $M=N=L=0$, $p=\infty$)
and Corollary \ref{cor:t1} (taken with $M=N=L=0$) combined with Lemma \ref{lem:finbridge} 
(specified to $p=\infty$, $s=0$), respectively.

To obtain the smoothness estimates, we must verify that the weak derivatives 
$\partial_{\t} \mathfrak{H}^{\ab}(\t,\vp)$ and $\partial_{\vp} \mathfrak{H}^{\ab}(\t,\vp)$
exist in the sense of \eqref{wder} and satisfy \eqref{grad}. In this case, $\texttt{v}$ 
is a complex measure in $[0,\infty]$, written $d\texttt{v}$, and
\[
\langle \texttt{v}, \mathfrak{H}^{\ab}(\t,\vp) \rangle
=
\int_{[0,\infty]} H_t^{\ab}(\t,\vp) \, d\texttt{v}(t).
\]
It is enough to consider the derivative with respect to $\t$. We claim that
\[
\partial_{\t} \langle \texttt{v}, \mathfrak{H}^{\ab}(\t,\vp) \rangle
=
\int_{[0,\infty]} \partial_{\t} H_t^{\ab}(\t,\vp) \, d\texttt{v}(t), \quad \t\ne\vp;
\]
observe that $\big\{ \partial_{\t} H_t^{\ab}(\t,\vp) \big\}_{t>0} \in \mathbb{X}$, $\t\ne\vp$, 
as can be seen from Proposition \ref{prop:JP} and the bound $\q \gtrsim (\t-\vp)^2$.
This claim would imply that for $\t\ne\vp$ the weak derivative 
$\partial_{\t} \mathfrak{H}^{\ab}(\t,\vp)$ exists and equals 
$\big\{ \partial_{\t} H_t^{\ab}(\t,\vp) \big\}_{t>0}$. To see that it then also satisfies \eqref{grad}, 
we first consider large $t$ and observe that the estimate
\begin{align*}
\big\| \partial_\t H_t^{\ab}(\t,\vp) \big\|_{L^\infty( (1,\infty), dt)}
	\lesssim 
\frac{1}{|\t-\vp|\mu_{\ab}(B(\t,|\t-\vp|))}, \qquad \t\ne\vp,
\end{align*}
follows from Corollary \ref{cor:tL} 
(specified to $M=L=0$, $N=W=1$, $p = \infty$). For small $t$, we have
\begin{align*}
\big\| \partial_\t H_t^{\ab}(\t,\vp) \big\|_{L^\infty( (0,1), dt)}
	\lesssim 
\frac{1}{|\t-\vp|\mu_{\ab}(B(\t,|\t-\vp|))}, \qquad \t\ne\vp,
\end{align*}
in view of Corollary \ref{cor:t1} (with $M=L=0$, $N=1$) and Lemma \ref{lem:finbridge} 
(taken with $W=1$, $p=\infty$, $s=1$).

To prove the claim, we verify that 
\begin{equation}\label{maxder}
\int_{\t_1}^{\t_2} \int_{[0,\infty]} \partial_\t H_t^{\ab}(\t,\vp) \, d\texttt{v}(t) \, d\t
=
\int_{[0,\infty]} H_t^{\ab}(\t_2,\vp) \, d\texttt{v}(t) - 
\int_{[0,\infty]} H_t^{\ab}(\t_1,\vp) \, d\texttt{v}(t),
\end{equation}
where $\t_1$ and $\t_2$ are chosen in $[0,\pi]$ so that $\vp\notin [\t_1,\t_2]$. Then the claim follows by 
differentiation with respect to $\t_2$. Finally, \eqref{maxder} follows from Fubini's theorem, since 
$\partial_\t H_t^{\ab}(\t,\vp)$ is continuous and bounded in $[0,\infty] \times [\t_1,\t_2]$.

\noindent
{\bf The case of $\boldsymbol{R_N^{\ab}(\t,\vp)}$.}
To prove the growth condition, it is enough to verify that 
\[
\big\| \partial_\t^N H_t^{\ab}(\t,\vp) \big\|_{L^1( t^{ N - 1} dt)}
\lesssim
\frac{1}{\mu_{\ab}(B(\t,|\t-\vp|))}, \qquad \t \ne \vp.
\]
This, however, is a consequence of Corollary \ref{cor:tL} (taken with $M=L=0$, $W=N$, $p=1$) and Corollary 
\ref{cor:t1} (with $M=L=0$) combined with Lemma \ref{lem:finbridge} (specified to $W=N$, $p=1$, $s=0$).

In order to show the gradient bound \eqref{grad}, it suffices to check that
\[
\Big\| \big| \nabla_{\! \t,\vp} \partial_\t^N H_t^{\ab}(\t,\vp) \big|
\Big\|_{L^1( t^{ N - 1} dt)}
\lesssim
\frac{1}{|\t-\vp|\mu_{\ab}(B(\t,|\t-\vp|))}, \qquad \t \ne \vp.
\]
This estimate follows by means of Corollary \ref{cor:tL} (applied with $M=0$, $p=1$) and Corollary 
\ref{cor:t1} (with $M=0$) together with Lemma \ref{lem:finbridge} (specified to $W=N$, $p=1$, $s=1$).

\noindent
{\bf{{The case of $\boldsymbol{\mathfrak{G}^{\ab}_{M,N}(\t,\vp)}$.}}}
The growth condition is a straightforward consequence of Corollary \ref{cor:tL} (with $L=0$, $W=2M + 2N$, 
$p=2$), Corollary \ref{cor:t1} (with $L=0$) and Lemma \ref{lem:finbridge} (taken with $W=2M + 2N$, $p=2$, 
$s=0$).

Next, we show the gradient estimate \eqref{grad}, which amounts to
\[
\Big\| \big|\nabla_{\! \t,\vp} \partial_\t^{N}\partial_t^M H_t^{\ab}(\t,\vp) \big|
\Big\|_{L^2(t^{2M + 2N - 1} dt)}
\lesssim 
\frac{1}{|\t-\vp|\mu_{\ab}(B(\t,|\t-\vp|))}, \qquad \t\ne\vp,
\]
where $\nabla_{\! \t,\vp}$ is taken in the weak sense.
This follows with the aid of  
Corollary \ref{cor:tL} (with $W=2M+2N$, $p=2$),
Corollary \ref{cor:t1} and Lemma \ref{lem:finbridge} (applied with $W=2M + 2N$, $p=2$, $s=1$); cf. the 
arguments given for the case $\mathfrak{H}^{\ab}(\t,\vp)$ above.

\noindent
{\bf{{The case of $\boldsymbol{K^{\ab}_{\phi}(\t,\vp)}$.}}}
The growth bound is a direct consequence of the assumption $\phi \in L^\infty(dt)$, Corollary \ref{cor:tL} 
(specified to $M=1$, $N=L=0$, $W=1$, $p=1$), Corollary \ref{cor:t1} (with $M=1$, $N=L=0$) and Lemma 
\ref{lem:finbridge} (taken with $W=1$, $p=1$, $s=0$).

Since $\phi$ is bounded, to prove the gradient estimate it is enough to verify that 
\[
\Big\| \big| \nabla_{\! \t,\vp} \partial_t H_t^{\ab}(\t,\vp) \big| \Big\|_{L^1(dt)}
\lesssim
\frac{1}{|\t-\vp|\mu_{\ab}(B(\t,|\t-\vp|))}, \qquad \t \ne \vp.
\] 
Now applying Corollary \ref{cor:tL} (with $M=1$, $W=1$, $p=1$ and either $N=1$, $L=0$ or $N=0$, $L=1$), 
Corollary \ref{cor:t1} (specified to $M=1$ and either $N=1$, $L=0$ or $N=0$, $L=1$) and Lemma 
\ref{lem:finbridge} (taken with $W=1$, $p=1$, $s=1$) we arrive at the desired bound.

\noindent
{\bf{{The case of $\boldsymbol{K^{\ab}_{\nu}(\t,\vp)}$.}}}
To show the growth condition it is enough, by the assumption \eqref{assum} concerning the measure $\nu$, 
to check that 
\begin{align*} 
\Big\| e^{t\left| \frac{\a+\b+1}{2} \right|} H_t^{\ab}(\t,\vp)  \Big\|_{L^\infty((1,\infty), dt)}
&\lesssim
\frac{1}{\mu_{\ab}(B(\t,|\t-\vp|))}, \qquad \t \ne \vp,\\
\|  H_t^{\ab}(\t,\vp) \|_{L^\infty((0,1), dt)}
&\lesssim
\frac{1}{\mu_{\ab}(B(\t,|\t-\vp|))}, \qquad \t \ne \vp.
\end{align*}
The first estimate above is an immediate consequence of Corollary \ref{cor:tLStieltjes} 
(applied with $N=L=0$). 
On the other hand, the remaining bound is just part of the growth condition for
$\mathfrak{H}^{\ab}(\t,\vp)$, which is already justified.

Taking \eqref{assum} into account, to verify the gradient estimate \eqref{grad}, it suffices to show that
\begin{align*} 
\left\| e^{t\left| \frac{\a+\b+1}{2} \right|} 
\big| \nabla_{\! \t,\vp} H_t^{\ab}(\t,\vp) \big| \right\|_{L^\infty((1,\infty), dt)}
&\lesssim
\frac{1}{|\t-\vp|\mu_{\ab}(B(\t,|\t-\vp|))}, \qquad \t \ne \vp, \\
\left\| \big| \nabla_{\! \t,\vp} H_t^{\ab}(\t,\vp) \big| 
\right\|_{L^\infty((0,1), dt)}
&\lesssim
\frac{1}{|\t-\vp|\mu_{\ab}(B(\t,|\t-\vp|))}, \qquad \t \ne \vp.
\end{align*}
Again, an application of Corollary \ref{cor:tLStieltjes} (with either $N=1$, $L=0$ or $N=0$, $L=1$) produces 
the first bound. The second one is contained in the proof of the gradient estimate for 
$\mathfrak{H}^{\ab}(\t,\vp)$.

The proof of Theorem \ref{thm:std} is complete.
\end{proof}

%%%%%%%%%%%%%%%%%%%%%%%%%%%%%%%%%%%%%%%%%%%%%%%%%%%%%%%%%%%%%%%%%%%%%%%%%%%%%%%%%%%%
\section{Calder\'on-Zygmund operators} \label{sec:CZ}
%%%%%%%%%%%%%%%%%%%%%%%%%%%%%%%%%%%%%%%%%%%%%%%%%%%%%%%%%%%%%%%%%%%%%%%%%%%%%%%%%%%%%

Let $\mathbb{B}$ be a Banach space and
suppose that $T$ is a linear operator assigning to each $f\in L^2(d\m_{\ab})$
a strongly measurable $\mathbb{B}$-valued function $Tf$ on $[0,\pi]$. 
Then $T$ is said to be a (vector-valued)
Calder\'on-Zygmund operator in the sense of the space $([0,\pi],d\m_{\ab},|\cdot|)$ associated with
$\mathbb{B}$ if
\begin{itemize}
\item[(A)] $T$ is bounded from $L^2(d\m_{\ab})$ to $L^2_{\mathbb{B}}(d\m_{\ab})$,
\item[(B)] there exists a standard $\mathbb{B}$-valued kernel $K(\t,\vp)$ such that
$$
Tf(\t) = \int_0^{\pi} K(\t,\vp) f(\vp)\, d\m_{\ab}(\vp), \qquad \textrm{a.a.}\; \t \notin \support f,
$$
for $f \in L^{\infty}([0,\pi])$.
\end{itemize}
Here integration of $\mathbb{B}$-valued functions is understood in Bochner's sense, and
$L^2_{\mathbb{B}}(d\m_{\ab})$ is the Bochner-Lebesgue space of all $\mathbb{B}$-valued 
$d\m_{\ab}$-square integrable functions on $[0,\pi]$.

It is well known that a large part of the classical theory of Calder\'on-Zygmund operators remains valid,
with appropriate adjustments, when the underlying space is of homogeneous type and the associated kernels
are vector-valued, see for instance \cite{RRT,RuTo}. In particular, if $T$ is a Calder\'on-Zygmund
operator in the sense of $([0,\pi],d\m_{\ab},|\cdot|)$ associated with a Banach space $\mathbb{B}$,
then its mapping properties in weighted $L^p$ spaces follow from the general theory. 

Let
$$
\mathcal{H}_t^{\ab}f(\t) = \int_0^{\pi} H_t^{\ab}(\t,\vp) f(\vp)\, d\m_{\ab}(\vp), \qquad t>0, \quad
	\t \in [0,\pi],
$$
be the Jacobi-Poisson semigroup. For $\ab > -1$ consider the following operators defined initially
in $L^2(d\m_{\ab})$.
\begin{itemize}
\item[(I)] 
The Jacobi-Poisson semigroup maximal operator
$$
\mathcal{H}_*^{\ab}f = \big\| \mathcal{H}_t^{\ab}f \big\|_{L^{\infty}(dt)}.
$$
\item[(II)] 
Riesz-Jacobi transforms of orders $N=1,2,\ldots$
$$
R_N^{\ab}f = \sum_{n=1}^{\infty} \Big| n + \frac{\a+\b+1}2\Big|^{-N} 
	\big\langle f,\P_n^{\ab}\big\rangle_{d\m_{\ab}} \, \partial_{\theta}^N\P_n^{\ab},
$$
where $\big\langle f,\P_n^{\ab}\big\rangle_{d\m_{\ab}}$ are the Fourier-Jacobi coefficients of $f$.
\item[(III)] 
Littlewood-Paley-Stein type mixed square functions
$$
g_{M,N}^{\ab}(f) = 
	\big\|\partial_{\theta}^{N}\partial_t^M\mathcal{H}_t^{\ab}f\big\|_{L^2(t^{2M+2N-1}dt)},
$$
where $M,N = 0,1,2,\ldots$ and $M+N>0$.
\item[(IV)] 
Multipliers of Laplace and Laplace-Stieltjes transform type
$$
M^{\ab}_{\mathfrak{m}}f = \sum_{n=0}^{\infty} \mathfrak{m}\bigg(\Big|n+\frac{\a+\b+1}2\Big|\bigg)
	\big\langle f,\P_n^{\ab}\big\rangle_{d\m_{\ab}} \P_n^{\ab},
$$
where either $\mathfrak{m}(z) = \int_0^{\infty} z e^{-tz} \phi(t)\, dt$ with $\phi \in L^{\infty}(dt)$
or $\mathfrak{m}(z) = \int_{(0,\infty)} e^{-tz} \, d\nu(t)$ with $\nu$ being a signed or complex Borel
measure on $(0,\infty)$ whose total variation satisfies \eqref{assum}.
\end{itemize}

The formulas defining $\mathcal{H}_*^{\ab}$ and $g_{M,N}^{\ab}$ are understood pointwise and are
actually valid for general functions $f$ from weighted $L^p$ spaces with Muckenhoupt weights.
This is because for such $f$ the integral defining $\mathcal{H}_t^{\ab}f(\t)$ is well defined and produces
a smooth function of $(t,\t)\in (0,\infty)\times [0,\pi]$, see \cite[Section 2]{NoSj}. The series
defining $R_N^{\ab}$ and $M_{\mathfrak{m}}^{\ab}$ indeed converge in $L^2(d\m_{\ab})$, 
which is clear in the case of $M_{\mathfrak{m}}^{\ab}$,
since the values of $\mathfrak{m}$ that occur here stay bounded.
For $R_N^{\ab}$ the convergence
follows by \cite[Lemma 3.1]{NoSj}, see the proof of \cite[Proposition 2.2]{NoSj} in the case of $R_N^{\ab}$.

As a consequence of Theorem \ref{thm:std} we get the following result.
\begin{thm} \label{thm:main}
Assume that $\ab > -1$. The Riesz-Jacobi transforms and the multipliers of Laplace and Laplace-Stieltjes
transform type are scalar-valued Calder\'on-Zygmund operators in the sense of the space
$([0,\pi],d\m_{\ab},|\cdot|)$. Furthermore, the Jacobi-Poisson semigroup maximal operator and the mixed
square functions can be viewed as vector-valued Calder\'on-Zygmund operators in the sense of
$([0,\pi],d\m_{\ab},|\cdot|)$, associated with the Banach spaces $\mathbb{B}=\mathbb{X}$ and
$\mathbb{B} = L^2(t^{2M+2N-1}dt)$, respectively.
\end{thm}

\begin{proof}
The standard estimates are provided in all the cases by Theorem \ref{thm:std}.
Thus it suffices to verify $L^2$ boundedness and kernel associations (conditions (A) and (B) above).
This, however, was essentially done in \cite[Section 3]{NoSj}, since the arguments given there are
actually valid for all $\ab > -1$ if combined with the estimates proved (in some cases implicitly)
in Section \ref{sec:kerest}.
An exception here are the Laplace and Laplace-Stieltjes type multipliers. But in these cases the
boundedness in $L^2$ is straightforward, and kernel associations are justified according to
the outline opening the proof of \cite[Proposition 2.3]{NoSj}, see \cite[Section 3, pp.\,732--733]{NoSj}.
Since all the necessary ingredients are
contained in \cite{NoSj} and in the present paper, we leave further details to interested readers.
\end{proof}

Denote by $A_p^{\ab}$, $1 \le p < \infty$, the Muckenhoupt classes of weights related to the space
$([0,\pi],d\m_{\ab},|\cdot|)$ (see \cite[Section 1]{NoSj} for the definition).
\begin{cor}
Let $\ab > -1$. The Riesz-Jacobi transforms and the multipliers of Laplace and Laplace-Stieltjes
type extend to bounded linear operators on $L^p(wd\m_{\ab})$, $w \in A_p^{\ab}$, $1<p<\infty$,
and from $L^1(wd\m_{\ab})$ to weak $L^1(wd\m_{\ab})$, $w \in A_1^{\ab}$. 
The same boundedness properties hold for the Jacobi-Poisson
semigroup maximal operator and the mixed square functions, viewed as scalar-valued sublinear operators.
\end{cor}

\begin{proof}
The part concerning $R_N^{\ab}$ and $M_{\mathfrak{m}}^{\ab}$ is a direct consequence of 
Theorem \ref{thm:main} and the general theory. The remaining part follows by Theorem \ref{thm:main}
and the arguments given in the proof of \cite[Corollary 2.5]{NoSj}.
\end{proof}

%%%%%%%%%%%%%%%%%%%%%%%%%%%%%%%%%%%%%%%%%%%%%%%%%%%%%%%%%%%%%%%%%%%%%%%%%%%%%%%%%%%%
\section{Exact behavior of the Jacobi-Poisson kernel} \label{sec:JPsharp}
%%%%%%%%%%%%%%%%%%%%%%%%%%%%%%%%%%%%%%%%%%%%%%%%%%%%%%%%%%%%%%%%%%%%%%%%%%%%%%%%%%%%%

We give another application of the representations in Proposition \ref{prop:JP},
which is interesting and important in its own right.
We will describe in a sharp way 
the behavior of the kernels $\H_t^{\ab}(\t,\vp)$ and $H_t^{\ab}(\t,\vp)$.
The result below extends sharp estimates for the Jacobi-Poisson kernel obtained in 
\cite[Theorem 5.2]{NoSj2} under the restriction $\ab \ge -1/2$.
\begin{thm} \label{thm:JPsharp}
Let $\ab > -1$. Then
\begin{align*}
H^{\ab}_t(\t,\vp) & \simeq \H^{\ab}_t(\t,\vp) \\ & \simeq \Big( t^2+ \t^2+\vp^2 \Big)^{-\a-1/2}
	\Big( t^2 + (\pi-\t)^2 + (\pi-\vp)^2 \Big)^{-\b-1/2} \frac{t}{t^2+(\t-\vp)^2},
\end{align*}
uniformly in $0 < t \le 1$ and $\t,\vp \in [0,\pi]$, and
$$
H_t^{\ab}(\t,\vp)  \simeq \exp\bigg( -t \frac{|\alpha+\beta+1|}{2} \bigg), \qquad
\H_t^{\ab}(\t,\vp)  \simeq \exp\bigg( -t \frac{\alpha+\beta+1}{2} \bigg),
$$
uniformly in $t \ge 1$ and $\t,\vp \in [0,\pi]$.
\end{thm}

To prove this we will need some technical results, one of which is Lemma \ref{lem:intest} (a).
Note that this lemma remains true if the integration is restricted to the subinterval
$(1/2,1]$. This follows from the structure of $d\Pi_{\nu}$ and the fact that the integrand is 
positive and increasing. 

\begin{lem}\label{lem:abcdest}
Let $\tau>0$ be fixed. Then
$$ 
\frac{1}{a^{\tau}} - \frac{1}{b^{\tau}} - \frac{1}{c^{\tau}} + \frac{1}{d^{\tau}} \gtrsim
\frac{(b \wedge c - a)^2 \wedge a^2}{a^{\tau+2}},
$$
uniformly in $0<a \le b,c \le d$ satisfying $a+d=b+c$.
\end{lem}

\begin{proof}
We can assume that $b\le c$. Then the right-hand side is independent of $c$ and $d$. In the left-hand side, 
we therefore replace $c$ and $d$ by $c+s$ and $d+s$, respectively, where $s\ge b-c$. By differentiating, we 
see that the function $s \mapsto -(c+s)^{-\tau} + (d+s)^{-\tau}$ is increasing. As a result, we need only 
consider the extreme case $s=b-c$, which means proving the lemma for $b=c$. 

Writing $h=b-a$, and letting $f(x)=x^{-\tau}$, the left-hand side is now the second difference 
$f(a)-2f(a+h)+f(a+2h)$, which equals $f''(\xi)h^2$ for some $\xi \in (a,a+2h)$. Now if $h>Ca$ for some large 
$C=C(\tau)$, the inequality of the lemma is trivial, since the term $a^{-\tau}$ will dominate in the 
left-hand side. But if $h\le Ca$, we have $f''(\xi) \simeq a^{-\tau-2}$, and the conclusion follows again.
\end{proof}

Let $\sigma > 1$ be fixed. Then one easily verifies that
\begin{equation} \label{lem62}
\big| x^{-\sigma}-y^{-\sigma}\big| \simeq \frac{|x-y|}{(x\vee y)(x \wedge y)^{\sigma}}, \qquad x,y>0.
\end{equation}

\begin{proof}[{Proof of Theorem \ref{thm:JPsharp}}]
We first prove the estimates for $\H_t^{\ab}(\t,\vp)$. Among the four ranges of the type parameters
distinguished in Proposition \ref{prop:JP}, it is enough to consider only two. Indeed, when
$\ab \ge -1/2$ the desired bounds are contained in \cite[Theorem 5.2]{NoSj2}, and the cases
$\b < -1/2 \le \a$ and $\a < -1/2 \le \b$ are essentially the same.
In what follows we denote for $t>0$ and $\t,\vp \in [0,\pi]$
$$
X := \frac{\st \svp}{\cosh\frac{t}2-\ct \cvp}, \qquad
Y := \frac{\ct \cvp}{\cosh\frac{t}2-\st \svp},
$$
and
$$
Z:= \frac{\sinh\frac{t}2}{\big(\cosh\frac{t}2-\ct \cvp\big)^{\a+1/2} 
	\big(\cosh\frac{t}2-\st \svp\big)^{\b+1/2}
	\big(\cosh\frac{t}2- \st \svp - \ct \cvp\big)}.
$$
Notice that $0 \le X,Y < 1$ and that $Z$ is comparable, uniformly in $0 < t \le 1$ and $\t,\vp \in [0,\pi]$,
with the expression describing the short-time behavior in Theorem \ref{thm:JPsharp}; see the proof
of \cite[Theorem 5.2]{NoSj2}. Moreover, $Z$ has the same long-time behavior as that
asserted for $\H_t^{\ab}(\t,\vp)$.
Thus that part of the statement of Theorem \ref{thm:JPsharp} which deals with $\H_t^{\ab}(\t,\vp)$ can be 
written simply as 
\begin{equation}\label{HZcomp}
\H_t^{\ab}(\t,\vp) \simeq Z, \qquad t>0, \quad \t,\vp \in [0,\pi].
\end{equation}

\noindent{\bf Case 1:} ${-1<\a<-1/2\le \b}.$ By Proposition \ref{prop:JP},
\begin{align*}
\H_t^{\ab}(\t,\vp) & = \iint -\partial_u \Puv \, \oa \, \pib \\
	& \quad + \iint \Puv \, \piu \, \pib \\
	&  \equiv I_1 + I_2.
\end{align*}
One finds that
the integral $I_1$ is dominated (up to a multiplicative constant) by its restriction to the subsquare
$(1/2,1]^2$ and that the essential contribution to $I_2$ comes from integrating over $(1/2,1]^2$. 
In view of Lemma \ref{lem:ximes}, the measures $|\Pi_{\a}(u)|\,du$ and $d\Pi_{\a+1}$ are comparable on
$(1/2,1]$, and we infer that
\begin{align*}
I_1 & \lesssim \sinh\frac{t}2 \st \svp \iint \frac{d\Pi_{\a+1}(u)\, \pib}{\big( \cosh\frac{t}2
	- u \st \svp - v \ct \cvp \big)^{\a+\b+3}}, \\
I_2 & \simeq \sinh\frac{t}2 \int \frac{\pib}{\big( \cosh\frac{t}2
	- \st \svp - v \ct \cvp \big)^{\a+\b+2}},
\end{align*}
uniformly in $t>0$ and $\t,\vp \in [0,\pi]$. Applying now Lemma \ref{lem:intest} (a) to $I_1$ twice,
first to the integral against $d\Pi_{\b}$, with the parameters $\nu=\b$, $\kappa=0$,
$\gamma = \a+\b+3$, $A=\cosh\frac{t}2-u\st \svp$, $B= \ct \cvp$, and then to the resulting
integral against $d\Pi_{\a+1}$, with the parameters $\nu=\a+1$, $\kappa=\b+1/2$, $\gamma = \a + 5/2$,
$D=\cosh\frac{t}2$, $A = \cosh\frac{t}2-\ct \cvp$, $B=\st \svp$, we arrive at the bound
$$
I_1 \lesssim X Z.
$$
Applying once again Lemma \ref{lem:intest} (a), this time to $I_2$ and with the parameters 
$\nu = \b$, $\kappa = 0$, $\gamma = \a+\b+2$, $A = \cosh\frac{t}2-\st \svp$, $B=\ct \cvp$, we get
$$
I_2 \simeq (1-X)^{-\a-1/2} Z.
$$

Estimating $I_1$ from below is slightly more subtle. Notice that
$$
I_1 = \sum_{\eta =\pm 1} \, \iint_{(0,1]^2} \big( \partial_u \Psi^{\ab}(t,\t,\vp,u,\eta v)
	- \partial_u \Psi^{\ab}(t,\t,\vp,-u,\eta v) \big) \, |\Pi_{\a}(u)|\, du\, \pib;
$$
here the integrand in each double integral is nonnegative, and the one corresponding 
to $\eta=1$ is dominating.
Thus restricting the set of integration to $(1/2,1]^2$ and making use of Lemma \ref{lem:ximes}, we write
\begin{align*}
I_1 & \gtrsim \iint_{(1/2,1]^2} \big( \partial_u \Psi^{\ab}(t,\t,\vp,u,v)
	- \partial_u \Psi^{\ab}(t,\t,\vp,-u,v) \big) \, d\Pi_{\a+1}(u)\, \pib \\
	& \simeq \sinh\frac{t}2 \st \svp \iint_{(1/2,1]^2} \Bigg[ 
	\frac{1}{\big(\cosh\frac{t}2-u\st \svp - v\ct \cvp\big)^{\a+\b+3}} \\
	& \quad 
	- \frac{1}{\big(\cosh\frac{t}2+u\st \svp - v\ct \cvp\big)^{\a+\b+3}}\Bigg] \, d\Pi_{\a+1}(u)\, \pib.
\end{align*}
Applying \eqref{lem62} to the expression in square brackets above, we get
\begin{align*}
I_1 & \gtrsim \iint_{(1/2,1]^2} \frac{\sinh\frac{t}2 (\st \svp)^2 \, u \, d\Pi_{\a+1}(u)\, \pib}
	{\big(\cosh\frac{t}2+u\st \svp - v\ct \cvp\big)\big(\cosh\frac{t}2-u\st \svp - v\ct \cvp\big)^{\a+\b+3}}\\
	& \gtrsim \iint_{(1/2,1]^2} \frac{\sinh\frac{t}2 (\st \svp)^2 \, d\Pi_{\a+1}(u)\, \pib}
	{\big(\cosh\frac{t}2+\st \svp - v\ct \cvp\big)\big(\cosh\frac{t}2-u\st \svp - v\ct \cvp\big)^{\a+\b+3}}.
\end{align*}
The last integral is comparable with an analogous integral over the larger square $[-1,1]^2$,
see the comment following Theorem \ref{thm:JPsharp}. 
Using now Lemma \ref{lem:intest} (a) twice, first for the integral against $d\Pi_{\b}$
(with the parameters $\nu=\b$, $\kappa = 1$, $\gamma = \a+\b+3$, $D = \cosh\frac{t}2+\st \svp$,
$A=\cosh\frac{t}2-u\st \svp$, $B = \ct \cvp$) and then with the resulting integral against
$d\Pi_{\a+1}$ (with $\nu=\a+1$, $\kappa = \b+1/2$, $\gamma = \a + 5/2$, $D=\cosh\frac{t}2$,
$A=\cosh\frac{t}2-\ct \cvp$, $B=\st \svp$), we arrive at the bound
$$
I_1 \gtrsim \frac{X^2}{X+1}\, Z \simeq X^2 Z.
$$

Summing up, we have proved that
$$
X^2 Z + (1-X)^{-\a-1/2} Z \lesssim \H_t^{\ab}(\t,\vp) \lesssim XZ + (1-X)^{-\a-1/2} Z,
$$
uniformly in $t>0$ and $\t,\vp \in [0,\pi]$, and \eqref{HZcomp} follows.

\noindent{\bf Case 2:} ${-1<\a,\b<-1/2}.$ In view of Proposition \ref{prop:JP}, 
\begin{align*}
\H_t^{\ab}(\t,\vp) & = \iint \partial_{u} \partial_{v} \Puv \, \oa \, \Pi_{\b}(v)\, dv \\  
& \quad + \iint - \partial_{u}  \Puv \, \oa \, \piv\\ 
& \quad + \iint - \partial_{v}  \Puv \, \piu \, \Pi_{\b}(v)\, dv\\
& \quad + \iint \Puv \, \piu \, \piv\\
& \equiv J_1 + J_2 + J_3 + J_4.
\end{align*}
Clearly, the main contribution to $J_4$ comes from the point $(u,v)=(1,1)$, and so 
$$
J_4 \simeq \Psi^{\alpha,\beta}(t,\theta,\varphi,1,1)
\simeq (1-X)^{-\a-1/2} (1-Y)^{-\b-1/2} Z
\le Z.
$$
To bound the remaining integrals from above, we proceed as in Case 1, obtaining
\begin{align*}
J_1 & \lesssim \iint \partial_{u} \partial_{v} \Puv \, d\Pi_{\a+1}(u)\, d\Pi_{\b+1}(v), \\
J_2 & \lesssim \int \partial_{u} \Psi^{\ab}(t,\t,\vp,u,1) \, d\Pi_{\a+1}(u), \\
J_3 & \lesssim \int \partial_{v} \Psi^{\ab}(t,\t,\vp,1,v) \, d\Pi_{\b+1}(v).
\end{align*}
Then applying repeatedly Lemma \ref{lem:intest} (a) with suitably chosen parameters, we get
$$
J_1 \lesssim XYZ \le Z, \qquad J_2 \lesssim X(1-Y)^{-\b-1/2}Z \le Z, 
\qquad J_3 \lesssim (1-X)^{-\a-1/2}YZ \le Z.
$$

To estimate $J_2$ and $J_3$ from below, we use the same trick as for $I_1$ in Case 1. 
By means of Lemma \ref{lem:ximes} and \eqref{lem62}, we can write
$$
J_2 \gtrsim \frac{\sinh\frac{t}2 \big(\st \svp\big)^2}{\cosh\frac{t}2 + \st \svp - \ct \cvp} \int _{(1/2,1]}
	\frac{d\Pi_{\a+1}(u)}{\big(\cosh\frac{t}2-u\st \svp - \ct \cvp\big)^{\a+\b+3}}.
$$
Then Lemma \ref{lem:intest} (a) shows that
$$
J_2 \gtrsim \frac{X^2}{X+1} (1-Y)^{-\b-1/2} Z \simeq X^2 (1-Y)^{-\b-1/2} Z.
$$
The case of $J_3$ is parallel, we have
$$
J_3 \gtrsim (1-X)^{-\a-1/2} \frac{Y^2}{Y+1} Z \simeq (1-X)^{-\a-1/2} Y^2 Z.
$$

Finally, we focus on the more delicate integral $J_1$. Observe that
$$
J_1 = \iint_{(0,1]^2} \sum_{\xi,\eta = \pm 1} \xi \eta \, \partial_u \partial_v 
	\Psi^{\ab}(t,\t,\vp,\xi u, \eta v) \,
	|\Pi_{\a}(u)|\, du\, |\Pi_{\b}(v)|\, dv.
$$
Restricting here the region of integration (the integrand is nonnegative, as we shall see in a moment)
and using Lemma \ref{lem:ximes}, we conclude
$$
J_1 \gtrsim \sinh\frac{t}2 \st \svp \ct \cvp \iint_{(1/2,1]^2} \bigg(\frac{1}{a^{\tau}}
	- \frac{1}{b^{\tau}} - \frac{1}{c^{\tau}} + \frac{1}{d^{\tau}} \bigg) \, d\Pi_{\a+1}(u)\, d\Pi_{\b+1}(v),
$$
where $\tau=\a+\b+4$, $a = \cosh\frac{t}2-u\st \svp - v\ct \cvp$, 
$b = \cosh\frac{t}2-u\st \svp + v\ct \cvp$,
$c = \cosh\frac{t}2+u\st \svp - v\ct \cvp$,
$d = \cosh\frac{t}2+u\st \svp + v\ct \cvp$.
Applying now Lemma \ref{lem:abcdest}, we get
$$
J_1 \gtrsim \sinh\frac{t}2 \st \svp \ct \cvp \iint_{(1/2,1]^2}
	\frac{(b \wedge c - a)^2 \wedge a^2}{a^{\a+\b+6}} \, d\Pi_{\a+1}(u)\, d\Pi_{\b+1}(v).
$$
Since 
$$
b \wedge c -a = 2 u \st \svp \wedge 2 v \ct \cvp \ge \st \svp \wedge \ct \cvp, \qquad u,v \ge 1/2,
$$
we can write
\begin{align*}
J_1 & \gtrsim \sinh\frac{t}2 \st \svp \ct \cvp \iint_{(1/2,1]^2} \frac{d\Pi_{\a+1}(u)\, d\Pi_{\b+1}(v)}
	{\big(\cosh\frac{t}2-u\st \svp - v\ct \cvp\big)^{\a+\b+6}} \\
	& \quad \times\bigg[ \Big(\cosh\frac{t}2-\st\svp - \ct \cvp \Big) \wedge
	\st \svp \wedge \ct \cvp \bigg]^2.
\end{align*}
Combining this with Lemma \ref{lem:intest} (a), we see that
$$
J_1 \gtrsim XY \bigg[1 \wedge \Big(\frac{X}{1-X}\Big)^2 \wedge \Big(\frac{Y}{1-Y}\Big)^2 \bigg]Z
\ge (X \wedge Y)^4 Z.
$$

Altogether, the above considerations justify the estimates
\begin{align*}
& \Big( (X\wedge Y)^4 + X^2(1-Y)^{-\b-1/2}  + (1-X)^{-\a-1/2}Y^2  
	+ (1-X)^{-\a-1/2}(1-Y)^{-\b-1/2}\Big) Z \\
& \qquad \lesssim  \H^{\ab}_t(\t,\vp) \lesssim Z,
\end{align*}
which hold uniformly in $t>0$ and $\t,\vp \in [0,\pi]$. From this, \eqref{HZcomp} follows.

We pass to the Jacobi-Poisson kernel $H_t^{\ab}(\t,\vp)$. Here we can assume that 
$\l:=\a+\b +1< 0$, since otherwise
the kernels $\H_t^{\ab}(\t,\vp)$ and $H_t^{\ab}(\t,\vp)$ coincide. Then
$$
H_t^{\ab}(\t,\vp) = \H_t^{\ab}(\t,\vp) + 2^{\l + 1} c_{\ab}\, \sinh\frac{\l t}2.
$$
The second term here is negative for $t>0$, so $H_t^{\ab}(\t,\vp) < \H_t^{\ab}(\t,\vp)$. Taking \eqref{HZcomp}
into account, we obtain the
short-time upper bound for $H_t^{\ab}(\t,\vp)$. Thus what remains to show is the lower bound and the 
long-time upper bound for $H_t^{\ab}(\t,\vp)$. 

We first claim that the lower short-time bound holds provided that $t>0$ is small enough. 
In view of the already justified estimates for $\H_t^{\ab}(\t,\vp)$, this will follow once we check that
$$
-2^{\l + 1} c_{\ab}\, \sinh \frac{\l t}{2} 
\le c \, \H_t^{\ab}(\t,\vp), \qquad
	0< t \le T_0,
$$
for some $T_0>0$ and some $c<1$. Notice that the hypergeometric series defining $F_4$ in \eqref{HF4} has 
nonnegative terms and that the zero-order term is $1$. Thus for $t>0$ and $\t,\vp \in [0,\pi]$
$$
\bigg(\frac{2}{e}\bigg)^{\l+1}\H_t^{\ab}(\t,\vp) + 2^{\l + 1} c_{\ab}\, \sinh \frac{\l t}{2} \ge
\bigg(\frac{2}{e}\bigg)^{\l+1} c_{\ab} \Bigg( \frac{\sinh\frac{t}2}{\big(\cosh\frac{t}2\big)^{\l+1}}
	+ e^{\l+1}\sinh \frac{\l t}2\Bigg).
$$
Now it suffices to ensure that, given $\l \in (-1,0)$, the function
$$
h(s) = \frac{\sinh s}{(\cosh s)^{\l+1}} + e^{\l+1}\sinh(\l s)
$$
satisfies $h(0)=0$ and $h'(0)>0$. This, however, is straightforward. The claim follows.

Next we show that the upper long-time bound for $H_t^{\ab}(\t,\vp)$ holds for $t \ge 1$
and that the lower counterpart is also true provided that $t \ge T_1$ with $T_1$ chosen large enough.
From the series representation,
$$
H_t^{\ab}(\t,\vp) = 2^{\l} c_{\ab}\, e^{-t|\l| /2} + \sum_{n=1}^{\infty} e^{-t(n+\l /2)}
	\P_n^{\ab}(\t)\P_n^{\ab}(\vp).
$$
The last series can be controlled by means of the bound $|\P_n^{\ab}(\t)| \lesssim n$, $n \ge 1$,
see \eqref{Jpolb}. More precisely, we have 
$$
\bigg| \sum_{n=1}^{\infty} e^{-t(n+\l /2)} \P_n^{\ab}(\t)\P_n^{\ab}(\vp)\bigg| \lesssim
	e^{-t/2} \sum_{n=1}^{\infty} n^2 e^{-t\left(n+\frac{\a+\b}2\right)} \lesssim e^{-t/2}, \qquad t \ge 1.
$$
Since $\a+\b > -2$ and $|\l|<1$, the conclusion follows.

To deal finally with the lower bound in the range $T_0 \le t \le T_1$, we use the semigroup property of 
$H_t^{\ab}$. For $T_0 \le t \le 2T_0$, we have
\[
H_t^{\ab}(\t,\vp) = 
\int_0^\pi H_{t/2}^{\ab}(\t,\psi) H_{t/2}^{\ab}(\psi,\vp) \, d\mu_{\ab}(\psi).
\]
Since $H_{t/2}^{\ab}(\t,\vp) \gtrsim 1$ in $[T_0,2T_0]\times [0,\pi]^2$ 
by the above, 
we conclude that also $H_t^{\ab}(\t,\vp)$ has a positive lower bound in the same set.
In a finite number of similar steps, we will reach $t=T_1$.

The proof of Theorem \ref{thm:JPsharp} is complete.
\end{proof}

\end{document}